\documentclass[11pt]{amsart}
\usepackage{amsmath,amssymb,amsfonts,amsthm, amscd,indentfirst}
\usepackage{amsmath,latexsym,amssymb,amsmath,
amscd,amsthm,amsxtra}
\usepackage{hyperref}\usepackage{url}
\usepackage{color}
\usepackage{pdfpages}
\usepackage{graphics}
\usepackage{enumitem}  
\DeclareMathOperator{\arccosh}{arccosh}
\newtheorem{theorem}{Theorem}[section]

\newtheorem{prop}{Proposition}[section]

\newtheorem{remark}{\textbf{Remark}}[section]

\def\rr{\mathbb{R}}
\def\ss{\mathbb{S}}
\def\hh{\mathbb{H}}
\def\bb{\mathbb{B}}

\def\tr{\mathrm{tr}}

\def\O{\Omega}
\def\p{\partial}

\def\a{\alpha}
\def\b{\beta}

\def\d{\delta}
\def\th{\theta}

\def\s{\sigma}

\def\p{\partial}

\def\S{{\Sigma}}
\def\<{\langle}
\def\>{\rangle}
\def\div{{\rm div}}
\def\n{\nabla}
\def\G{\Gamma}
\def\ode{\bar{\Delta}}
\def\on{\bar{\nabla}}
\def\De{\Delta}
\def\vp{\varphi}
\def\R{{\mathbb R}}
\def\ep{\epsilon}

\def\w{\mathcal{W}}

\def\mm{\mathbb{M}}

\numberwithin{equation} {section}

\begin{document}

\title[Uniqueness of stable capillary hypersurfaces]{Uniqueness of stable capillary hypersurfaces in a ball}
\author{Guofang Wang}
\address{ Universit\"at Freiburg,
Mathematisches Institut,
Ernst-Zermelo-Str. 1,
79104 Freiburg, Germany}
\email{guofang.wang@math.uni-freiburg.de}
\author{Chao Xia}
\address{School of Mathematical Sciences\\
Xiamen University\\
361005, Xiamen, P.R. China}
\email{chaoxia@xmu.edu.cn}
\thanks{Part of this work was done while CX was visiting the mathematical institute of Albert-Ludwigs-Universit\"at Freiburg. He would like to thank the institute for its hospitality. Research of CX is  supported in part by NSFC (Grant No. 11501480) and  the Natural Science Foundation of Fujian Province of China (Grant No. 2017J06003), and the Fundamental Research Funds for the Central Universities (Grant No. 20720180009).
}

\begin{abstract}
    In this paper we prove that any immersed stable capillary hypersurfaces in a ball in space forms are totally umbilical. 
    Our result also provides a proof
of a conjecture proposed by Sternberg-Zumbrun in {\it J Reine Angew Math 503 (1998), 63--85}.
 We also prove a Heintze-Karcher-Ros type inequality for hypersurfaces with free boundary 
 in a ball, which, together with the new Minkowski formula,
 yields a new proof of Alexandrov's Theorem for embedded CMC hypersurfaces in a ball with free boundary.
     \end{abstract}

\date{}
\maketitle

\medskip

\tableofcontents

\section{
Introduction}

Let $(\bar M^{n+1}, \bar g)$ be an oriented $(n+1)$-dimensional Riemannian manifold and $B$ be a smooth 
compact domain in $\bar M$ with non-empty boundary $\p B$. 
We are interested in capillary hypersurfaces,  namely    minimal or constant mean curvature (CMC) hypersurfaces 
in $B$ with boundary  on $\p B$ and intersecting $\p B$ at a constant angle $\theta \in (0,\pi)$. Minimal or CMC hypersurfaces 
with free boundary, namely, intersecting $\p B$ orthogonally, are special and  important examples of capillary hypersufaces.
Capillary hypersurfaces are critical points of some geometric variational functional under
certain volume constraint. It  has a very long history.  It was Thomas Young who  first considered  capillary surfaces mathematically  in 1805 and introduced the mathematical concept of
 mean curvature of a surface \cite{Young}. His work was followed by Laplace and later by Gauss. For the reader who are interested in the history of capillary surfaces, we refer to an article of  Finn-McCuan-Wente \cite{FMW}. See also Finn's book \cite{Finn2} for a survey about the mathematical theory of capillary surfaces.

 The stability of minimal or CMC hypersurfaces plays an important role in differential geometry. 
  For closed hypersurfaces (i.e. compact without boundary), there is a classical uniqueness result proved by Barbosa-do Carmo \cite{BDC} and Barbosa-do Carmo-Eschenburg \cite{BDCE}:
 {\it any stable closed CMC hypersurfaces in space forms are geodesic spheres.} 
In this paper we are concerned with stable capillary hypersurfaces in a ball in space forms. It is known that totally geodesic balls and spherical caps are stable and even area-minimizing. In fact, these are only isoperimetric hypersurfaces in a ball 
which was  first proved by Burago-Maz'ya\footnote{The authors would like to thank Professor Frank Morgan for  this information.} \cite{BM} and later also by Bokowsky-Sperner \cite{BoSp}  and Almgren \cite{Almgren}. Ros-Souam \cite{RS} showed that totally geodesic balls and spherical 
caps are capillary stable. Conversely,
the uniqueness problem was first studied by Ros-Vergasta  \cite{RV} for minimal or CMC hypersurfaces in free boundary case, i.e., $\theta=\frac \pi 2$ and later Ros-Souam \cite{RS}  for general capillary ones. Their works
 have been followed
 by many mathematicians. 
 Comparing to the uniqueness result for stable closed hypersurfaces \cite{BDC, BDCE}, there is a natural and long standing open problem on the uniqueness of stable capillary hypersurfaces since the work of Ros-Vergasta and Ros-Souam:
 
\medskip

 {\it Are any immersed stable capillary hypersurfaces in a ball in space forms totally umbilical?}
 
 \medskip
The main objective of this paper is to 
give a complete answer to this open problem.
For convenience, we discuss in the introduction mainly on the case of hypersurfaces in {\it a Euclidean ball with free boundary} and give a brief discussion about general capillary hypersurfaces in a ball in any space forms later.
 It is surprising that this problem leaves quite open 
 except in the following special cases. 
 \begin{enumerate}
 \item When $n\ge 2$, $H=0$ and $\theta=\frac \pi 2$, i.e., in the case of minimal hypersurfaces with free boundary,
 Ros-Vergasta gave an affirmative answer in \cite{RV} (1995). 
 \item When $n=2$, $H=const.$ and $\theta =\frac \pi 2$, i.e., in the case of 2-dimensional CMC surfaces with free boundary, Ros-Vergata \cite{RV} and 
 Nunes \cite{Nunes} (2017) gave an affirmative answer. 
  See also the work of E. Barbosa \cite{Barbosa}.
 \end{enumerate}
 

The stability for CMC hypersurfaces  is defined by using variations with a volume constraint. For minimal ones we also use this stability, which is also 
 called the weak stability. 
 A general way to utilize the stability condition is to find admissible test functions.
 For a volume constraint problem, such an  admissible function $\varphi$
 should satisfy $\int_M \varphi dA=0$, i.e., its average is zero.
In the work  
 of  Barbosa-do Carmo \cite{BDC} 
 for closed hypersurfaces,
the test function is defined by using the classical Minkowski formula, that is, for a closed immersion $x:M\to \R^{n+1}$,
 \begin{eqnarray}
  \label{eq0.1}
  &&\int_M H \< x, \nu \>  dA= n \int_M  dA,
 \end{eqnarray}
 where $H$ is the mean curvature and $\nu$ is the outward unit normal of $M$.
 In fact, in this case the test function is $\psi=n-H \< x, \nu \> $. The Minkowski formula \eqref{eq0.1} implies that this is an admissible  function.
 For a hypersurface $M$ in a ball with free boundary, Ros-Vergasta  \cite{RV} obtained the following Minkowski formula 
 \begin{eqnarray}\label{eq0.2}
  &&|\partial M| = n  |M| -\int_M H \< x, \nu \> dA.
 \end{eqnarray}
 Unlike \eqref{eq0.1}, the Minkowski formula \eqref{eq0.2} provides a relationship among three  geometric quantities, the area of the boundary $\partial M$,
 the area of $M$ and an integral involving the mean curvature.
 It is this complication that makes free boundary problems more difficult than problems for closed hypersurfaces.
In the minimal case, the Minkowski formula \eqref{eq0.2} relates only two geometric quantities, since the term involving the mean curvature vanishes. The proof of 
Result (1) relies
on this fact.

There is another way to find admissible test functions, which is called  a Hersch type balancing
argument. This argument is extremely useful, especially in two-dimensional problems, see for example the work of Li-Yau \cite{LY} and Montiel-Ros \cite{MR}.
 Using such an argument, together with the Minkowski formula \eqref{eq0.2},
Ros-Vergasta proved in \cite{RV} the following partial result.

{\it If $M\subset \bar \bb^3$ is an immersed compact stable CMC surface with
free boundary, then $\p M$ is embedded and the only possibilities are
\begin{enumerate}[label=(\roman*)]
 \item $M$ is a totally geodesic disk;
 \item $M$ is a spherical cap;
 \item $M$ has genus 1 with at most two boundary components.
\end{enumerate}}
Case (iii) was excluded very recently by Nunes \cite{Nunes} by using a new stability criterion and a modified Hersch type balancing argument.
See also the work of E. Barbosa  \cite{Barbosa} without using the modified Hersch type balancing argument. Therefore, when  $n=2$
this open problem was solved. This is result (2).

There are several partial results on the uniqueness of stable CMC hypersurfaces in a Euclidean ball with free boundary, see e.g., \cite{RV, Marinov0, LiXiong, Barbosa}. 



We remark that there are many embedded or non-embedded non-spherical examples.
In fact, for any constant $H > 0$ there
is a piece of an unduloid of mean curvature $H$ in the Euclidean unit ball $\bb^{n+1}$ with free boundary, which is however unstable. In fact,   Ros \cite{Ros2}
proved that neither catenoid nor unduloid pieces, which intersect  $\p\bb^{n+1}$ orthogonally, are stable.
The following uniqueness result classifies all stable immersed CMC hypersurfaces with free boundary in a Euclidean ball.
\begin{theorem}\label{thm0.1}
 Any stable immersed CMC hypersurface with free boundary in a Euclidean ball
is either a totally geodesic ball or a spherical cap.
\end{theorem}

One of crucial ingredients to prove this result is a new  Minkowski type formula. For an immersion $x: M\to \bar \bb^{n+1}$ with free boundary, we establish
a  weighted Minkowski formula 
\begin{equation}
 \label{eq0.3}
 n\int_M V_a dA  =\int _M H \<X_a, \nu\> d A,
\end{equation}
which is one of a family of Minkowski's formulae proved in Section 3. Here  $a\in \R^{n+1}$ is any constant vector field, $V_a$ and  $X_a$  are defined by
\begin{equation*}
 V_a:=\<x,a\>, \qquad 
 X_a:= \<x, a\>x-\frac 12 (1+|x|^2)a.
\end{equation*}
The key feature of $X_a$ is its conformal Killing property. For the details about $V_a$ and $X_a$  see  Section \ref{sec3} below.
Different to \eqref{eq0.2},
this new Minkowski formula \eqref{eq0.3} gives a relation between  two (weighted) geometric quantities. More important is that there is no boundary integral in this 
 new Minkowski formula. It is clear to see from \eqref{eq0.3} that $nV_a-H\<X_a, \nu\>$ is an admissible test function for the stability for any $a\in \R^{n+1}$.
These admissible functions play an essential role in the proof of  Theorem \ref{thm0.1}.
 
\medskip
 
 It is interesting that our proof works for stable CMC hypersurfaces with free boundary in $\bar \bb^{n+1}$ with a singular set of sufficiently low Hausdorff  dimension
 and therefore gives a proof of a conjecture proposed by
 Sternberg-Zumbrun (\cite{SZ1} p. 77). As an application of their stability formula (Theorem 2.2 in \cite{SZ1}), which they called Poincar\'e inequality for stable
 hypersurfaces with a singular set with Hausdorff measure $\mathcal{H}^{n-2}=0$, 
 they proved in \cite{SZ1} (Theorem 3.5) that any local minimizer of perimeter under the volume constraint in $\bar \bb^{n+1}$ is either 
 a totally geodesic ball or a regular graph over $\p \bb^{n+1}$, 
 provided that $H=0$ or 
 \begin{equation}\label{eq0.5} \int_M \< x, \nu  \> d\mathcal{H}^{n} <0.
   \end{equation}
Condition \eqref{eq0.5} is equivalent to $n\mathcal{H}^{n}(M)<\mathcal{H}^{n-1}(M\cap \p\bb^{n+1})$. 
This condition is almost the same as that in one of results of Ros-Vergasta (Theorem 8 in \cite{RV}). 
They conjectured that \eqref{eq0.5} holds always for 
stable hypersurfaces with boundary
and all local minimizers  in $\bar \bb^{n+1}$ are regular. 
Here we prove more, namely all local minimizers  in $\bar \bb^{n+1}$ are totally geodesic balls or spherical caps. 
 
 \begin{theorem}\footnote{In view of this result Sternberg-Zumbrun asked in their new  paper \cite{SZ3}, whether volume-constrained local minimizers in a convex domain 
remain regular in arbitrary dimension $n$ and not just for $n\le 7$. For a further discussion, see \cite{SZ3}. }
  \label{thm0.3}
  Let $\O$ be a local minimizer of perimeter with respect to fixed volume in $\bar \bb^{n+1}$. Then $M=\overline{\p \Omega \cap \bb^{n+1}}$ is  
  the intersection of 
  $\bar \bb^{n+1}$ with either a plane through the origin or a sphere.
  
 \end{theorem}

 We remark that, E. Barbosa \cite{Barbosa}  also gave a proof of the  conjecture proposed by Sternberg-Zumbrun \cite{SZ1} about the regularity of the local minimizers in $\bar \bb^{n+1}$.
 
 The minimal or CMC hypersurfaces with free boundary  attract much attention of many mathematicians. In 80's there are many  existence results
 obtained from geometric variational methods, see for example, 
 \cite{Struwe, GJ1, Jost, Wang, BK}. The corresponding regularity problem has been studied by Gr\"uter-Jost \cite{GJ2}. 
 Recently one of inspiring work is a series of papers of Fraser-Schoen
 \cite{FS1, FS2, FS3}
 about minimal hypersurfaces with free boundary in a  ball and the first Steklov eigenvalue. See also \cite{Brendle, ChenJ, Volkmann, Lm, Pacard, Freidin, Ambrozio}.
 Our research  on the stability on CMC hypersurfaces are motivated by these results.
 
There are many interesting  properties of closed surfaces in a space form  that  are valid also for surfaces with free boundary. However, in many cases the proof for the case of surfaces with free boundary is quite different and becomes more difficult, while in other cases the  counterpart for surfaces with free boundary is still open. It means that the free boundary problems for surfaces
 are in general more difficult.
 Here we just mention several good examples. Comparing to the result of Montiel-Ros \cite{MR}: {\it Any minimal torus immersed in ${\mathbb S}^3$ by the first eigenfunctions is the Clifford torus}, Fraser-Schoen \cite{FS2} took much more effort to obtain: {\it any minimal annulus with free boundary, which is immersed by the first Steklov eigenvalue, is the critical catenoid.} 
 While the Lawson conjecture about uniqueness of embedded torus in $\mathbb S^3$ was solved recently by Brendle \cite{Brendle_L} with a clever use of the maximum principle on a two-point function, the free boundary version of the Lawson conjecture is still open. See \cite{FL} and also \cite{N}, where it was claimed without providing a proof. Even if any minimal surfaces with free boundary with index 4 is the critical catenoid is also open.

\medskip

Let us turn to the general case, the capillary hypersurfaces in a ball (in space forms). There are only partial results. 
 See for example the work of Ros-Souam \cite{RS} mentioned already above, and also \cite{Souam, Marinov0, LiXiong}.
Our approach to prove Theorem \ref{thm0.1} is powerful enough to work for immersed capillary hypersurfaces in a ball in any space forms after establishing appropriate weighted Minkowski formulae, see Propositions \ref{MinkEucl} and \ref{Minksf}. 
In other words, we can give a complete affirmative answer to the open problem mentioned above.  
 \begin{theorem}\label{thm0.2}
 Any stable immersed capillary hypersurface  in a ball in space forms is totally umbilical.
\end{theorem}
For this theorem, though the ideas of proof are essentially the same as the one for Theorem \ref{thm0.1}, the proof becomes more involving.

By going through the proof, we see that our approach also works for closed hypersurfaces. Namely, we provide a new proof of the uniqueness results of Barbosa-do Carmo and Barbosa-do Carmo-Eschenburg mentioned above, see Remark 3.2 and Remark 3.3 below. Furthermore, our approach works for the corresponding exterior problem.  To be precise, we are able to prove the following
 \begin{theorem}\label{thm0.4}
Any compact stable immersed capillary hypersurface  outside a ball in space forms is totally umbilical.
\end{theorem}

\begin{remark}
From the proof we can easily see that we {\bf do not} need  the immersed hypersurface is contained {\em in} or {\em outside }a ball, but only need the assumption $x(\p M)\subset \p B$.
\end{remark}
 
There are many interesting uniqueness results on stable capillary hypersurfaces within other types of domains, e.g., 
 a wedge, a slab, a cone, a cylinder or  a half space, see e.g.
 \cite{AS, ChoeKoiso, Park, Morabito, Marinov, LiXiong2, Lopez, Lopez2, RR, Rosales, Vogel}.

 
\medskip
 

There are other important uniqueness results concerning capillary hypersurfaces. One is Hopf type theorem which says any CMC {\it $2$-sphere} in $\rr^3$ is a round sphere. Nitsche \cite{N} proved that any disk type capillary surface in $\bb^3$ is either a totally geodesic disk or a spherical cap by using Hopf type argument, see also Fraser-Schoen \cite{FS3} for recent development.
Another is Alexandrov type theorem which says that any {\it embedded} CMC closed hypersurface is a round sphere. For capillary hypersurfaces, if it is embedded with its boundary
$\p M$ lying in a half sphere, then Ros-Souam \cite{RS} (Proposition 1.2) showed that it is either a totally geodesic ball or a spherical 
cap by Alexandrov's reflection method.  
 
In the last section we will give a new proof of the Alexandrov type theorem  \cite{RS} for CMC hypersurface with free boundary by using integral method in the spirit as Reilly \cite{Reilly} and Ros \cite{Ros}. The key ingredients are the new Minkowski formula as well as a Heintze-Karcher-Ros type inequality we will establish.
This is another objective of this paper.
For such an inequality we also use the weight function $V_a$.


The Heintze-Karcher-Ros inequality for an embedded closed  hypersurface $\S$ of positive mean curvature in $\R^{n+1}$ is 
\begin{eqnarray}\label{HKR0}
\int_{\S} \frac{1}{H} dA\ge \frac{n+1}{n}  \int_\O  d\O,
\end{eqnarray}
 where $\O$ is the enclosed body by $\S$. Equality in \eqref{HKR0} holds if and only if $\Sigma$ is a round sphere.
\eqref{HKR0} is a sharp inequality for hypersurfaces of
positive mean curvature inspired by a classical inequality of Heintze-Karcher \cite{HK}.
In 1987, Ros \cite{Ros} provided a proof of the above inequality by using a remarkable
Reilly formula (see \cite{Reilly}), and applied it to show Alexandrov's rigidity theorem for
high order mean curvatures. Recently, Brendle \cite{Brendle2} established  such an inequality 
in a large class of warped product spaces, including the space forms and the (Anti-de Sitter-)Schwarzschild
manifold. A geometric flow method, which is quite different from Ros' proof, was
used by Brendle. Motivated by Brendle's work, new Reilly type formulae have been
established by the second named author and his collaborators
in \cite{QX, LX, LX2}. These formulae will be used to establish the following Heintze-Karcher-Ros inequality: 

{\it for an embedded hypersurface $\S$ lying in a half ball $B_+$ in any space forms with its boundary $\p \S\subset \p B_+$, there holds
\begin{eqnarray}\label{HKR1}
\int_{\S} \frac{V_a}{H} dA\ge \frac{n+1}{n}  \int_\O V_a d\O,
\end{eqnarray}
where $\O$ is the enclosed body by $\S$ and $\p B_+$. Equality in \eqref{HKR1} holds if and only if $\S$ is totally umbilical and intersects $\p B_+$ orthogonally.}

\noindent See Theorem \ref{HK} below. 
The Alexandrov rigidity theorem for embedded CMC hypersurfaces with free boundary follows from this inequality and the Minkowski formula \eqref{eq0.3}. 
We believe that there is a sharp version of Heintze-Karcher-Ros type inequality for hypersurfaces in a ball, whose equality case is achieved by capillary hypersurfaces with  a fixed contact angle $\theta\in (0, \pi)$.
\medskip

The remaining part of this paper is organized as follows. In Section \ref{sec2} we review the definition and basic properties of capillary hypersurfaces. Since we are concerned with the immersions,
a suitable notion of volume and the so-called wetting area  functional is needed to study capillary hypersurfaces. 
In Section \ref{sec3} we give a proof of Theorem \ref{thm0.2}  for capillary hypersurfaces in a ball in $\R^{n+1}$ after  establishing the Minkowski formula \eqref{Mink1}.
Theorem \ref{thm0.1} is a special case of Theorem \ref{thm0.2}. The same proof works for singular hypersurfaces, hence we have Theorem \ref{thm0.3}.
In Section \ref{sec4}  we provide a detailed proof of  Theorem \ref{thm0.2}  for capillary hypersurfaces in a ball in ${\mathbb H}^{n+1}$ and sketch  
a proof for capillary hypersurfaces in a ball in ${\mathbb S}^{n+1}$. For the corresponding exterior problem,  we sketch its proof at the end of Section \ref{sec4}.
In Section \ref{sec5}, we prove the Heintze-Karcher-Ros type inequality and  the Alexandrov theorem for hypersurfaces in a ball with free boundary.

\medskip

\section{Preliminaries on Capillary Hypersurfaces}
\label{sec2}

Let $(\bar M^{n+1}, \bar g)$ be an oriented $(n+1)$-dimensional Riemannian manifold and $B$ be a smooth 
compact domain in $\bar M$ that is diffeomorphic to a Euclidean ball. Let $x: (M^n, g)\to B$ be an  
isometric immersion of an orientable $n$-dimensional compact manifold $M$ with boundary $\p M$ into $B$  that maps ${\rm int} M$ into ${\rm int} B$ and $\p M$ into $\p B$.

We denote by $\bar \n$, $\bar \Delta$ and $\bar \n^2$ the gradient, the Laplacian and the Hessian on $\bar M$ respectively, while by $\n$, $\Delta$ and $\n^2$ the gradient, the Laplacian and the Hessian on $M$ respectively. 
We will use the following terminology for four normal vector fields.
We choose one of the unit normal vector field along $x$ and denote it by $\nu$. 
We denote  by $\bar N$ the unit outward normal to $\p B$ in $B$ and $\mu$ be the unit outward normal to $\p M$ in $M$. 
Let $\bar \nu$ be the unit normal to $\p M$ in $\p B$ such that the bases $\{\nu, \mu\}$ and $\{\bar \nu, \bar N\}$ have the same orientation in the
normal bundle of $\p M\subset \bar M$.
See Figure 1. 
Denote by $h$ and $H$ the second fundamental form and the mean curvature of the immersion $x$ respectively. Precisely,
$h(X, Y)= \bar g(\bar \n_X \nu, Y)$ and $H=\tr_g(h).$
For constant mean curvature hypersurfaces which are our main concern, we always choose $\nu$ to be one of the unit normal vector fields so that $H\ge 0$.
\begin{figure} 
\vspace{8mm}
\includegraphics[width=0.6\linewidth]{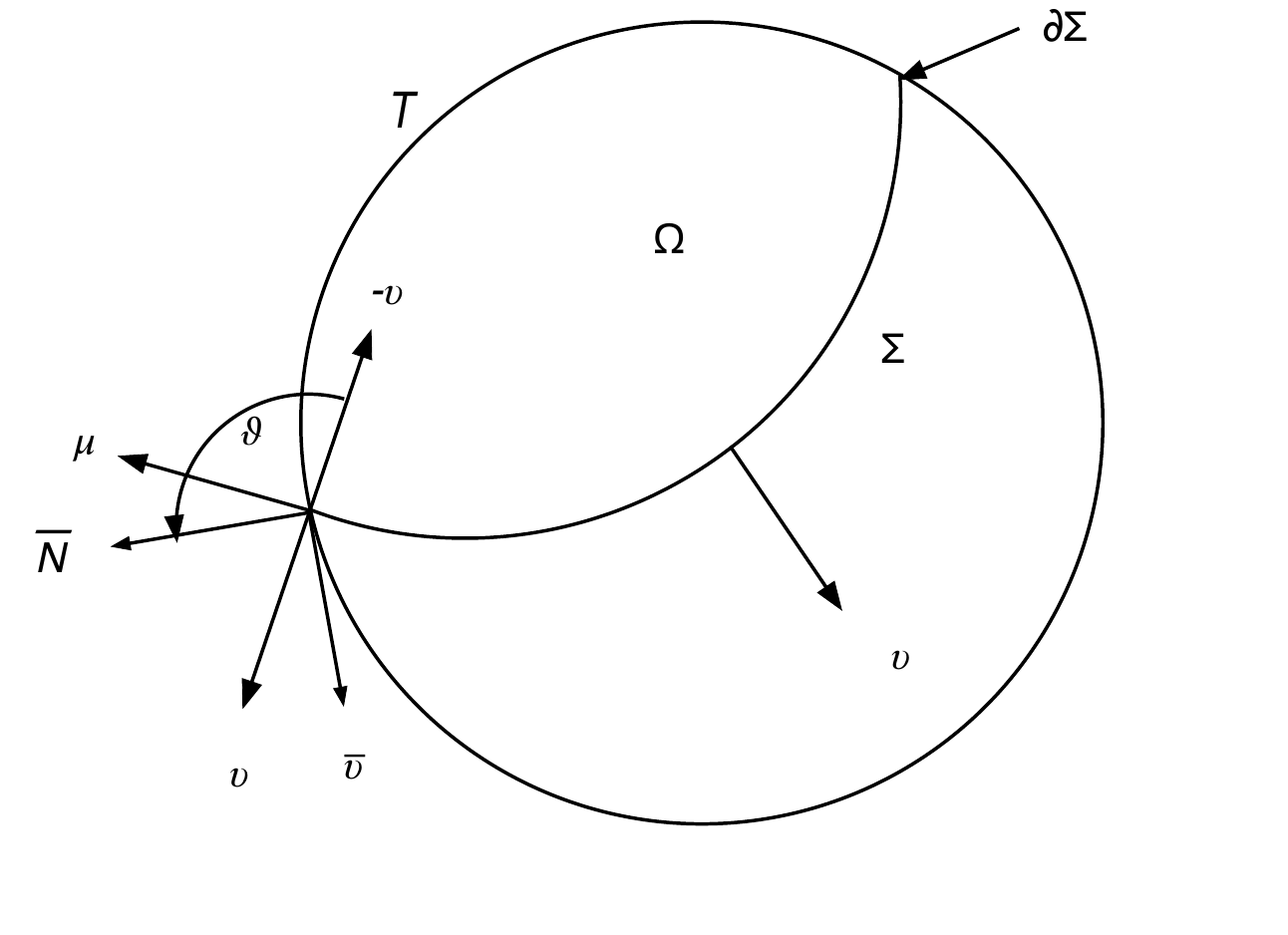}
 \label{fig1} 
 \caption{$\Sigma=x(M)$  and $\p \Sigma=x(\p M)$}
 \end{figure}


Since in this paper we consider immersions, we need to introduce generalized definitions of area, volume and a wetting area  for an isometric immersion. 
For embedded hypersurfaces, these generalized definitions are certainly equivalent to the usual definitions (See \cite{RV, RS}).

By an admissible variation of $x$ we mean a differentiable map $x: (-\epsilon, \epsilon)\times M\to B\subset \bar M$ such that $x(t, \cdot): M\to B$ is an immersion satisfying $x(t, {\rm int} M)\subset{\rm int} B$ and $x(t, \p M)\subset\p B$ for every $t\in (-\ep, \ep)$ and $x(0, \cdot)=x$.
For this variation, 
the area  functional             $A: (-\ep, \ep)\to\rr$ and the volume functional $V: (-\ep, \ep)\to\rr$ are defined by 
\begin{eqnarray*}
&&A(t)=\int_{ M} dA_t,\\
&&V(t)=\int_{[0,t]\times M} x^*dV_{\bar M},
\end{eqnarray*}
where $dA_t$ is the area element of $M$  with respect to the metric induced by $x(t, \cdot)$ and $dV_{\bar M}$ is the volume element of $\bar M$.
A variation is said to be volume-preserving if $V(t)=V(0)=0$ for each $t\in (-\ep, \ep)$. 
Another area functional, which is called wetting area functional, $W(t): (-\ep, \ep)\to\rr$ is defined by 
$$W(t)=\int_{[0,t]\times \p M} x^*dA_{\p B},$$
where $dA_{\p B}$ is the area element of $\p B$.

Fix a real number $\theta\in (0, \pi)$. The energy functional  $E(t): (-\ep, \ep)\to\rr$ is defined by
$$E(t)=A(t)-\cos \theta \, W(t).$$
The first variation formulae of $V(t)$ and $E(t)$ for an admissible variation with  a variation vector field $Y=\frac{\p}{\p t}x(t,\cdot)|_{t=0}$ are given by
\begin{eqnarray*}
&&V'(0)=\int_M \bar g(Y, \nu)dA,\\
&&E'(0)=\int_M H\bar g(Y, \nu)dA+\int_{\p M} \bar g(Y, \mu-\cos \theta \, \bar \nu)ds,
\end{eqnarray*}
where $dA$ and $ds$ are the area element of $M$ and $\p M$ respectively, see e.g. \cite{RS}. 

An immersion $x: M\to B$ is said to be {\it capillary} if it is a critical point of the energy function $E$ for any volume-preserving variation of $x$. 
It follows from the above  first variation formulae that $x$ is capillary if and only if $x$ has constant mean curvature and $\p M$ intersects $\p B$ at the constant angle $\theta$. We make a convention on the choice of $\nu$ to be
the opposite direction of mean curvature vector so that the  mean curvature of a spherical cap  is positive. Under this convention, along $\p M$, the angle between $-\nu$ and $\bar N$ or equivalently between $\mu$ and $\bar \nu$ is everywhere equal to $\th$ (see Figure 1). To be more precise, in the normal bundle of $\p M$, we have the following relations:
 \begin{eqnarray} 
&&\mu=\sin \th \, \bar N+\cos \th\, \bar \nu,  \label{mu0}\\&&\nu=-\cos \th \, \bar N+\sin \th\, \bar \nu.\label{nu0}
\end{eqnarray}

 For each smooth function $\vp$ on $M$ with $\int_M \vp dA_M=0$, there exists an admissible volume-preserving variation of $x$ with the variation vector field
 having $\varphi \nu$ as normal part (see  \cite{RS}, page 348).
When $x$ is a capillary hypersurface, for an admissible volume-preserving variation with respect to $\vp$, the second variational formula of $E$ is given by
\begin{eqnarray}\label{stab-ineq}
&&E''(0)=\int_M -\vp(\De \vp+(|h|^2+\overline{{\rm Ric}}(\nu,\nu))\vp) dA+\int_{\p M} \vp(\n_\mu \vp-q \vp)ds.
\end{eqnarray}
Here
$$q=\frac{1}{\sin \th}h^{\p B}(\bar \nu, \bar \nu)+\cot \th \, h(\mu,\mu),$$
$\overline{\rm Ric}$ is the Ricci curvature tensor of $\bar M$,  
and $h^{\p B}$ is the second fundamental form of $\p B$ in $\bar M$ given by $h^{\p B}(X, Y)= \bar g(\bar \n_X \bar N, Y)$, see e.g. \cite{RS}.

A capillary hypersurface is called stable if $E''(0)\ge 0$ for all volume-preserving variations, that is, \begin{eqnarray*}
E''(0)\ge 0,\qquad \forall \vp\in \mathcal{F}:=\left\{\vp\in C^\infty(M)| \int_M \vp dA=0\right\}.
\end{eqnarray*}

The following proposition is a well-known and fundamental  fact for capillary hypersurfaces when $\p B$ is umbilical in $\bar M$.  
 \begin{prop} \label{lemma1} Assume  $\p B$ is umbilical in $\bar M$. Let $x: M\to B$ be an immersion whose boundary $\p M$ intersects $\partial B$ at a constant angle $\th$. Then $\mu$ is a principal direction of $\p M$ in $M$. Namely, $h(e, \mu)=0$ for any $e\in T(\p M)$. In turn, \begin{eqnarray*}
\bar \n_\mu \nu= h(\mu, \mu)\mu.
\end{eqnarray*}
\end{prop}
\begin{proof}
For $e\in T(\p M)$, by using \eqref{mu0} and \eqref{nu0}, we have
\begin{eqnarray*}
h(e, \mu)&=& \bar g(\bar \n_e \nu, \mu)=\bar g(\bar \n_e (-\cos \th \, \bar N+\sin \th \, \bar \nu), \sin \th \, \bar N+\cos \th \, \bar \nu)
\\&=& -\bar g(\bar \n_e \bar N,  \bar \nu)= - h^{\p B}(e, \bar \nu)=0.\end{eqnarray*}
\end{proof}

\medskip

\section{Capillary Hypersurfaces in a Euclidean Ball}
\label{sec3}

In this section, we consider the case $(\bar M, \bar g)=(\rr^{n+1}, \delta)$ and $B=\bar \bb^{n+1}$ is the Euclidean unit ball (in our notation, $\bb^{n+1}$ is the Euclidean unit open ball).
In this case, $\overline{\rm Ric}\equiv 0$, $h^{\p\bb}=g^{\p\bb}$ and $\bar N(x)=x$.
Abuse of notation, we use $x$ to denote the position vector in $\rr^{n+1}$. We use $\<\cdot, \cdot\>$ to denote the Euclidean inner product.

\subsection{A new Minkowski type formula in $\rr^{n+1}$}\

In this subsection we establish a new Minkowski type formula, which is very powerful  for hypersurfaces in $\bb^{n+1}$ with free boundary  or 
intersecting $\p \bb^{n+1}$ with a constant angle.

We first consider a  conformal Killing vector field.
For each constant vector field $a\in \mathbb{R}^{n+1}$, define a corresponding smooth vector field $X_a$ in $\rr^{n+1}$ by
\begin{eqnarray}\label{ConfKillv}
X_a=\<x, a\>x-\frac12(|x|^2+1)a.
\end{eqnarray}

Define $f:\R^{n+1}\backslash {(0,-1)} \to \R^{n+1}$
by
\begin{equation*}
f(u,v)=\frac {2(u,0)+(|u|^2+v^2-1) e_{n+1}}{|u|^2 +(1+v)^2},\end{equation*}
where $(u,v)\in \R^{n+1}=\R^n\times \R$ and $e_{n+1}=(0,1)$.
One can check that $f$ maps  $\R^{n+1}_+ \to {\mathbb B}^{n+1}$ and  $\p \R^{n+1}_+ \to \ss^n$.
Moreover $f$ is conformal. In fact
\begin{equation*}
f^*(\d_{{\mathbb B}^{n+1}})=\frac 4{(|u|^2+(1+v)^2)^2} \d_{\R^{n+1}_+}.\end{equation*}
If one transfers the free boundary problem in $\bb^{n+1}$ to the free boundary problem in $\R^{n+1}_+$ with the pull back metric $f^*(\d_{{\mathbb B}^{n+1}})$, one obtains
an equivalent problem. The vector field $X_a$ with $a=e_{n+1}$ is the push-forward of
 the radial vector field (or the position vector field) $(u,v)$ with respect to the origin in $\R^{n+1}_+$, which is usually important in such problems. This is the way 
 we found that this vector field should be useful in the capillary problems. From this observation, it is clear that $X_a$ is conformal Killing and tangential to $\partial \bb^{n+1}$. Namely, we have the 
 following two simple but crucial  properties of $X_a$.
\begin{prop}\label{ConfK}\
$X_a$ is a conformal Killing vector field and its restriction on $\p \bb^{n+1}$ is a tangential vector field on $\p \bb^{n+1}$, i.e.,
\begin{itemize}
\item[(i)] $X_a$ is a conformal Killing vector field in $\rr^{n+1}$ with $\mathcal{L}_{X_a} \bar g= \<x, a\>\bar g$, namely,
\begin{eqnarray}\label{confKill}
\frac12\big[\bar \n_i (X_a)_j+ \bar \n_j (X_a)_i\big]=\<x, a\>\delta_{ij}.
\end{eqnarray}
\item[(ii)] $X_a|_{\p \bb}$ is a tangential vector field on $\p\mathbb{B}$. I.e.,
\begin{eqnarray}\label{x5}
&&\<X_a, x\>|_{\p\bb}=0.
\end{eqnarray}
\end{itemize}
\end{prop}
\begin{proof}  It is a well-known fact and one can check by a direct computation. 
\end{proof}

{{\begin{remark} The conformal Killing property of $X_a$ is well-known in conformal geometry. For each $a\in \rr^{n+1}$, $X_a$ generates a $1$-parameter family of conformal automorphism of $\bb^{n+1}$ onto itself, see \cite{LY}, page 274.  The restriction of $X_a$ to $\ss^n$ gives a conformal Killing vector field on $\ss^n$ generating an associated $1$-parameter family of conformal automorphism of $\ss^{n}$, which has been widely used in differential geometry and conformal geometry, see e.g. \cite{B,CY, MN, MR}. {This vector field was used by Fraser and Schoen in their study of free boundary to show the result mentioned in the Introduction about the first Steklov eigenvalue \cite{S_Talk}.}
 We also realized that this vector field has already been used in the capillary problems implicitly  by Ros-Vergasta \cite{RS} and explicitly by Marinov \cite{Marinov0} in $2$-dimension and Li-Xiong \cite{LiXiong} in any dimensions.\end{remark}}}

Utilizing the conformal Killing vector field $X_a$, we show the following Minkowski type formula.
\begin{prop}\label{MinkEucl}  Let $x: M\to \bar \bb^{n+1}$ be an  isometric immersion into the Euclidean unit ball, whose boundary $\p M$ intersects $\p \bb^{n+1}$ at a constant angle $\th \in (0, \pi)$. Let $a\in \rr^{n+1}$ be a constant vector field and $X_a$ be defined by \eqref{ConfKillv}. Then
\begin{eqnarray}\label{Mink1}
\int_M n \<x+\cos \th \, \nu, a\> dA=\int_M H\<X_a, \nu\> dA.
\end{eqnarray}
\end{prop}

\begin{proof}
Denote by $X_a^T$ the tangential projection of $X_a$ on $M$. 
Let $\{e_\a\}_{\a=1}^{n}$ be an orthonormal frame on $M$. 
We claim that \begin{eqnarray}\label{confKill2}
\frac12\left[\n_\a (X_a^T)_\b+\n_\b (X_a^T)_\a\right]= \<x, a\>g_{\a\b}-h_{\a\b}\<X_a,\nu\>.
\end{eqnarray}
Here $\n_\a (X_a^T)_\b :=\<\n_{e_\a} X_a^T, e_{\b}\>$.
In fact, 
\begin{eqnarray*}
\n_\a (X_a^T)_\b&=& \<\bar\n_{e_\a} X_a^T, e_{\b}\>\\
&=&\<\bar\n_{e_\a} X_a, e_{\b}\>- \<\bar\n_{e_\a} (\<X_a, \nu\>\nu), e_{\b}\>
\\&=&\bar\n_\a (X_a)_\b-\<X_a, \nu\>\<\bar\n_{e_\a} \nu, e_{\b}\>
\\&=& \bar\n_\a (X_a)_\b-h_{\a\b}\<X_a, \nu\>.
\end{eqnarray*}
By using \eqref{confKill}, we get the claim. 

Taking trace of $\eqref{confKill2}$ with respect to the induced metric $g$ and integrating over $M$, we have
 \begin{eqnarray}\label{min1}
&&\int_M n \<x, a\> - H \<X_a,\nu\> dA
=\int_M \div_M (X_a^T) dA
= \int_{\p M} \<X_a^T, \mu\> ds.
\end{eqnarray}
Note that on $\p M$, $\bar N=x$ and $X_a=\<x, a\>x-a$. By using \eqref{mu0}, \eqref{nu0} and \eqref{x5},  
we deduce \begin{eqnarray*}
\<X_a^T, \mu\>&=&\<X_a, \mu\>=\<X_a, \sin \th \, \bar N+\cos \th \, \bar \nu\>=\cos \th\, \<X_a, \bar \nu\>\\&=&\cos \th\, (\<x, a\>\<x, \bar \nu\>-\<a, \bar \nu\>)=-\cos \th\, \<a, \bar \nu\>.
\end{eqnarray*}
It follows from \eqref{min1}
\begin{eqnarray}\label{min2}
&&\int_M n\<x, a\>-H\<X_a, \nu\> dA=-\cos \th\int_{\p M} \<\bar \nu, a\>ds.
\end{eqnarray}
When $\theta =\frac{\pi}{2}$, i.e., if we are in the free boundary case, the Minkowski formula \eqref{Mink1} follows already from \eqref{min2}.
For the general case, we claim
\begin{eqnarray}\label{AS1}
n \int_M \<\nu, a\>dA= \int_{\p M}\<\bar \nu, a\>ds.
\end{eqnarray}
It is easy to see that the Minikowski formula \eqref{Mink1} follows from the claim and \eqref{min2}.

It remains to show this claim.  It has been shown in \cite{AS} that
\begin{eqnarray}\label{AS}
n \int_M \<\nu, a\>dA= \int_{\p M} \<x, \mu\>\<\nu, a\>-\<x, \nu\>\<\mu, a\>ds.
\end{eqnarray}
For the convenience of reader, we give a proof of \eqref{AS}.
Set $Z_a=\<\nu, a\>x- \<x, \nu\>a$. Then 
\begin{eqnarray*}
\div_M [(Z_a)^T]&=& [h(a^T, x^T)+\<\nu, a\> (n-\<x, \nu\>H)]-[h(x^T, a^T)-\<x, \nu\>\<\nu, a\>H]
\\&=&n\<\nu, a\>.
\end{eqnarray*}
Then \eqref{AS} follows by integration by parts.
From \eqref{mu0} and \eqref{nu0}, we deduce \begin{eqnarray*}
&&\<x, \mu\>\<\nu, a\>-\<x, \nu\>\<\mu, a\>\\
 &=&\sin \th\< -\cos \th \,\bar N+\sin \th\, \bar \nu, a\>+\cos\th\<\sin \th \, \bar N+\cos \th\, \bar \nu, a\>
 \\&=&\<\bar \nu, a\>.
\end{eqnarray*}
Therefore, we get the claim \eqref{AS1} and the proof is completed.
\end{proof}

\begin{remark} \label{rmk3.2} For the free boundary problem, i.e., $\theta =\pi/ 2$, we obtain the Minkowski formula discussed in the Introduction:
\begin{eqnarray}\label{Mink1_0}
n\int_M \<x, a\> dA=\int_M H\<X_a, \nu\> dA.
\end{eqnarray}
We  remark that \eqref{Mink1_0} holds also for any compact hypersurfaces without boundary in $\R^{n+1}$ with the same proof, just ignoring the boundary integral.
To our best knowledge it is also  new for any compact hypersurfaces without boundary and we believe that it has its own interest.
 
\end{remark}

Minkowski formula \eqref{Mink1} plays a crucial role in the proof of uniqueness  of stable capillary hypersurfaces in a Euclidean ball in the next subsection.
Its further interesting applications will be presented in Section \ref{sec5}.

 \subsection{Uniqueness of stable capillary hypersurfaces in a Euclidean ball}\

\begin{prop}\label{prop3.3} Let $x: M\to\bar \bb^{n+1}$ be an  isometric immersion into the Euclidean unit ball, whose boundary $\p M$ intersects $\p \bb^{n+1}$ 
at a constant angle $\th \in (0, \pi)$. Let $a\in \rr^{n+1}$ be a constant vector field. Then along $\p M$, 
\begin{eqnarray}
\bar \n_\mu \<x+\cos \th\,\nu, a\> &=& q \<x+\cos \th\,\nu, a\>, \label{bdry1}
\\\bar \n_\mu \<X_a, \nu\>&=& q \<X_a, \nu\>,\label{bdry2}
\end{eqnarray}
where \begin{eqnarray}\label{q1}
q=\frac{1}{\sin \th}+\cot \th \ h(\mu,\mu).
\end{eqnarray}
\end{prop}

\begin{proof} Using Proposition \ref{lemma1},
\begin{eqnarray*}
&&\bar \n_\mu \<x+\cos \th\,\nu, a\>
= \<\mu+\cos \th\, h(\mu, \mu)\mu, a\>=q \sin \th \<\mu, a\>.
\end{eqnarray*}
On the other hand, using \eqref{mu0} and \eqref{nu0},
\begin{eqnarray*}
\<x+\cos \th\,\nu, a\>&=& \<x+\cos \th (-\cos \th \bar N+\sin \th \bar \nu), a\>\\&=&\sin \th \<\sin \th \bar N+\cos \th \bar \nu, a\>= \sin \th \<\mu, a\>.
\end{eqnarray*}
Thus we get \eqref{bdry1}.

Using the definition \eqref{ConfKillv} of $X_a$ and again Proposition \ref{lemma1},
\begin{eqnarray*}
\bar \n_\mu \<X_a, \nu\>&=&\<\bar \n_\mu X_a, \nu\>+ \<X_a,  \bar \n_\mu \nu\>
\\&=&\<\bar \n_\mu (\<x, a\>x-\frac12(|x|^2+1)a), \nu\>+h(\mu, \mu)\<X_a, \mu\>
\\&=&\<\<\mu, a\>x+\<x, a\>\mu -\<x, \mu\>a, \nu\>+h(\mu, \mu)\<\<x, a\>x-a, \mu\>
\\&=&-\cos \th\<\mu, a\>-\sin\th \<\nu, a\>+h(\mu, \mu)(\sin \th \<x, a\>-\<\mu, a\>).
\end{eqnarray*}
Note that $x=\bar N= \sin \th \, \mu-\cos \th\,  \nu$ and in turn $\mu= \frac{1}{\sin \th}x+ \cot \th\,  \nu$, we deduce further
\begin{eqnarray*}
\bar \n_\mu \<X_a, \nu\>&=&-\cot \th (1+h(\mu, \mu)\cos \th)\<x, a\>- \frac{1}{\sin \th}(1+h(\mu, \mu)\cos \th) \<\nu, a\>
\\&=&-q (\cos \th \<x, a\>+\<\nu, a\>).
\end{eqnarray*}
On the other hand, 
\begin{eqnarray*}
\<X_a, \nu\>|_{\p M}&=&\<x, a\>\<x, \nu\>- \<a, \nu\>= -(\cos \th\<x, a\>+ \<\nu, a\>). 
\end{eqnarray*}
\eqref{bdry2} follows.
\end{proof}

\begin{prop} \label{lem3.1} Let $x: M\to \rr^{n+1}$ be an isometric immersion into the Euclidean space. Let $a\in \rr^{n+1}$ be a constant vector field. The following identities hold along $M$:
\begin{eqnarray}
&&\De x=-H\nu,\label{eq-x}
\\&&\De \frac12|x|^2=n-H\<x, \nu\>,\label{eq-x2}
\\&&\De \nu=\n H-|h|^2\nu,\label{eq-nu}
\\&&\De\<x, \nu\>=\<x, \n H\>+ H-|h|^2\<x, \nu\>, \label{eq-xnu}
\\&&\De \<X_a, \nu\>=\<X_a, \n H\>+\<x, a\>H-|h|^2\<X_a,\nu\>-n\<\nu, a\>.\label{eq-Xnu}
\end{eqnarray}
\end{prop}
\begin{proof} Equations \eqref{eq-x}--\eqref{eq-xnu} are well-known. 
We now prove \eqref{eq-Xnu}.
First,
\begin{eqnarray*}
\De \<X_a, \nu\>&=&\<\De X_a, \nu\>+2\<\n X_a, \n \nu\>+\<X_a, \De \nu\>.
\end{eqnarray*}
Using the definition \eqref{ConfKillv} of $X_a$, \eqref{eq-x} and \eqref{eq-x2}, we  see
\begin{eqnarray*}
\<\De X_a, \nu\>&=& \<\De (\<x, a\>x-\frac12(|x|^2+1)a), \nu\>\\
&=&\big\<-H\<\nu, a\>x-\<x, a\>H\nu- (n-H\<x, \nu\>)a, \nu\big\>
\\&=&- H\<x, a\>-n\<\nu, a\>.
\end{eqnarray*}
Also, 
\begin{eqnarray*}
\<\n X_a, \n \nu\>&=& \<\n (\<x, a\>x-\frac12(|x|^2+1)a), \n \nu\>\\
&=&\<e_\a, a\> h(e_\a, x^T)+\<x, a\>H- \<x, e_\a\>h(e_\a, a^T)
\\&=& H\<x, a\>.
\end{eqnarray*}
Using \eqref{eq-nu}, 
\begin{eqnarray*}
&&\<X_a, \De \nu\>=  \<X_a, \n H\>-|h|^2\<X_a, \nu\>.
\end{eqnarray*}
Combining above, we get \eqref{eq-Xnu}.
\end{proof}

\begin{prop}\label{lem3.2}Let $x: M\to\bar \bb^{n+1}$ be an  isometric immersion into the Euclidean unit ball, whose boundary $\p M$ intersects $\p \bb^{n+1}$ 
at a constant angle $\th \in (0, \pi)$. 
For each constant vector field $a\in \rr^{n+1}$ define
 $$\varphi_a=n \<x+\cos\th\,\nu, a\>-H\<X_a,\nu\>$$ along $M$. Then $\varphi_a$ satisfies
 \begin{eqnarray}
\int_M \varphi_a dA&=&0, \label{eq-phi1}\\
\n_\mu \varphi_a-q\varphi_a&=&0.\label{eq-phi2}
 \end{eqnarray}
If, in addition, that $M$ has constant mean curvature, then $\varphi_a$ satisfies also
\begin{eqnarray}\label{eq-phi}
&&\Delta \varphi_a+|h|^2\varphi_a=(n|h|^2-H^2)\<x,a\>.
\end{eqnarray}
\end{prop}

\begin{proof} \eqref{eq-phi1} and \eqref{eq-phi2} follow from   Propositions \ref{MinkEucl} and \ref{prop3.3} respectively. If $H$ is constant, Proposition \ref{lem3.1} implies
\begin{eqnarray*}
 (\Delta+|h|^2) \<x,a\> &=& |h|^2 \<x,a\> -H\<\nu,a\>,\\
  (\Delta +|h|^2) \<X_a,\nu\>  &=& H \<x,a\>- n \<\nu,a\>,\\
   (\Delta +|h|^2) \<\nu,a\> &=&0.
\end{eqnarray*}
Then \eqref{eq-phi} follows.
\end{proof}

Now we prove the uniqueness for stable capillary hypersurfaces in a Euclidean ball.
\begin{theorem}\label{thm1}
Assume $x: M\to \bar \bb^{n+1}$ is an immersed stable capillary hypersurface in the Euclidean unit ball $\bb^{n+1}$  with constant mean curvature $H \ge 0$ and constant contact angle $\th\in (0, \pi)$. Then $x$ 
 is either a totally geodesic ball or a spherical cap.
\end{theorem}
\begin{proof}
The stability condition states as
\begin{eqnarray}\label{stab-eq1}
&&-\int_M \varphi (\Delta \varphi+|h|^2\varphi)dA+\int_{\p M} \varphi(\n_\mu \varphi- q\varphi) ds\ge 0
\end{eqnarray}
for all function $\varphi\in \mathcal{F}$, where $q$ is given by \eqref{q1}.

For each constant vector field $a\in \rr^{n+1}$, we consider $\varphi_a$, which is defined in Proposition \ref{lem3.2}.
Proposition \ref{lem3.2} implies that 
 $\varphi_a \in \mathcal{F}$ and
is an admissible function for testing stability.
Inserting \eqref{eq-phi1} and \eqref{eq-phi} into the stability condition \eqref{stab-eq1}, we get
\begin{eqnarray}\label{stab-eq2}
&&\int_M \left(n\<x+\cos\th\,\nu, a\>-H\<X_a,\nu\>\right)\<x,a\>(n|h|^2-H^2) \, dA\le 0\hbox{ for any }a\in \mathbb{R}^{n+1}.
\end{eqnarray}
We take $a$ to be the $n+1$ coordinate vectors $\{E_i\}_{i=1}^{n+1}$ in $\R^{n+1}$, and add \eqref{stab-eq2} for all $a$ to get
\begin{eqnarray}\label{stab-eq3}
&&\int_M \left(n|x|^2+n\cos\th\<x, \nu\>-\frac12(|x|^2-1)H\<x,\nu\>\right)(n|h|^2-H^2)\, dA\le 0.
\end{eqnarray}
Here we have used $$\sum_{i=1}^{n+1}\<x,E_{i}\>X_{E_i}=\frac12(|x|^2-1)x.$$

Now, if $H=0$ and $\th=\frac{\pi}{2}$, \eqref{stab-eq3} gives  $\int_M |x|^2|h|^2 dA \le 0$, which implies that $h\equiv 0$, i.e., $x: M\to \bar \bb^{n+1}$ is totally geodesic. 
This gives a new proof of a result of Ros-Vergasta \cite{RV}.

When $H\not =0$ or $\th \neq \frac{\pi}{2}$, the proof does not follow from  \eqref{stab-eq3} directly. 
In fact the term 
\begin{equation}\label{add_eq}\gamma:=n|x|^2+n\cos\th\<x, \nu\>-\frac12(|x|^2-1)H\<x,\nu\>
 \end{equation}
may have no definite sign.
In order to handle this problem, we introduce the following  function 
$$\Phi=\frac12(|x|^2-1)H-n(\<x,\nu\>+\cos \th).$$
Using \eqref{eq-x2} and \eqref{eq-xnu}, one can check that $\Phi$ satisfies
\begin{eqnarray}\label{stab-eq4}
\De \Phi=(n|h|^2-H^2)\<x,\nu\>.
\end{eqnarray}
Since $|x|^2=1$ and $\<x, \nu\>=-\cos \th$ on $\p M$, we have $\Phi=0$ on $\p M$. Consequently,
\begin{eqnarray}\label{eq-Phi}
\int_M \De \frac12\Phi^2=\int_{\p M} \Phi \n_\mu\Phi \,dA=0.
\end{eqnarray}
Adding \eqref{eq-Phi} to \eqref{stab-eq3} and using \eqref{stab-eq4}, we obtain
\begin{eqnarray*}
0&\ge &
\int_M \left(n(|x|^2+\cos \th \<x, \nu\>)-\frac12(|x|^2-1)H\<x,\nu\>\right)(n|h|^2-H^2)+ \De \frac12\Phi^2  \, dA
\\&=&\int_M \left(n(|x|^2+ \cos \th \<x, \nu\>)-\frac12(|x|^2-1)H\<x,\nu\>\right)(n|h|^2-H^2)+ \Phi\De\Phi+|\n \Phi|^2  \, dA
\\&=&\int_M n|x^T|^2(n|h|^2-H^2)+|\n \Phi|^2  \, dA
\\&\ge &0,
\end{eqnarray*}
where $x^T$ is the tangential part of $x$.
The last inequality holds since $n|h|^2\ge H^2$ which follows from Cauchy's inequality.
It follows that 
 $|x^T|^2(n|h|^2-H^2)=0$ on $M$ and $\n \Phi=0$. The latter implies that $\Phi$ is a constant. This fact, together with \eqref{stab-eq4}, implies that
 $\< x, \nu\> (n|h|^2-H^2)=0$ on $M$. Together with $|x^T|^2(n|h|^2-H^2)=0$, it implies that $|x|^2(n|h|^2-H^2)=0$ on $M$.  Hence we have 
 $n|h|^2-H^2=0$ on $M$, which means that $M$ is umbilical and  is a spherical cap. The proof is completed.
 \end{proof}

\begin{remark} In the case of free boundary, i.e., $\cos \theta =0$, 
 Barbosa \cite{Barbosa} proved that $\<x, \nu\> $ has a fixed sign, namely, 
  $\<x, \nu\> \le 0 $ in $M$ in our notation. By our convention of the choice of $\nu$, we have $H>0$. It is not clear whether $\gamma$ has a sign from these information. However, one can show 
  the non-negativity of $\gamma$ with the help of $\Phi$ used in the proof as follows.
  In this case, by \eqref{stab-eq4} we have $\Delta \Phi\le 0$ in $M$ and $\Phi=0$ on $\p M$, which implies that $\Phi \ge 0$ by the maximum principle, and hence $-\<x, \nu\> \Phi \ge 0$ in $M$. 
  It follows that 
  $$\gamma \ge n\<x, \nu\>^2 -\frac 12 (|x|^2-1) H\<x, \nu\> = -\<x, \nu\> \Phi \ge 0.$$
 Therefore, in the case of free boundary, with the help of  the non-negativity of $\gamma$, we can get $n|h|^2-H^2=0$ from \eqref{stab-eq3}.
 \end{remark}
 
 \begin{remark} 
Since the new Minkowski formula holds also for closed hypersurfaces in $\R^{n+1}$, (Remark \ref{rmk3.2}),  the above proof for the stability of capillary
surfaces works without any changes for closed hypersurfaces. This means that we give a new proof of the result of Barbosa-do Carmo \cite{BDC} mentioned above.
This works also for closed hypersurfaces in space forms. See the next section.
\end{remark}
 

\medskip

\section{Capillary Hypersurfaces in a Ball in Space Forms}\label{sec4}

In this section we handle the case when $\bar M$ is a space form $\hh^{n+1}$ or $\ss^{n+1}$ and $B$ is a ball in $\bar M$. 
Since these two cases are quite similar, we will prove the hyperbolic case and indicate the minor modifications for the spherical case in Subsection 4.3 below.

\subsection{A new Minkowski type formula in $\hh^{n+1}$}\

Let $\hh^{n+1}$ be the simply connected hyperbolic space with curvature $-1$. We use here the  Poincar\'e ball model, which  is given by
\begin{eqnarray}\label{Poincare}
\hh^{n+1}=\left(\bb^{n+1}, \bar g=e^{2u} \delta\right), \quad e^{2u}=\frac{4}{(1-|x|^2)^2}.
\end{eqnarray}
One can also use other models. The advantage to use the  Poincar\'e ball model for us is that for this model it is relatively easy to find the corresponding conformal Killing vector field $X_a$.

In this section we use $\delta$ or $\<\cdot, \cdot\>$ to denote the Euclidean metric 
and the Cartesian coordinate in $\bb^{n+1}\subset \rr^{n+1}$.
Sometimes we also  represent the hyperbolic metric, in terms of the polar coordinate with respect to the origin, as $$\bar g=dr^2+\sinh^2 rg_{\ss^n}.$$ 
We use $r=r(x)$ to denote the hyperbolic distance from the origin and denote $V_0=\cosh r$.
It is easy to verify that  \begin{eqnarray}\label{hypfunc}
V_0=\cosh r=\frac{1+|x|^2}{1-|x|^2},  \quad \sinh r=\frac{2|x|}{1-|x|^2}.
\end{eqnarray}
 The position function $x$, in terms of polar coordinate, can be represented by \begin{eqnarray}\label{rad-conf0}
x=\sinh r \p_r. 
\end{eqnarray}
It is well-known that $x$ is a conformal Killing vector field with \begin{eqnarray}\label{rad-conf}
\bar\n x=V_0 \bar g.
\end{eqnarray}

Let $B^\hh_{R}$ be a ball in $\hh^{n+1}$ with hyperbolic radius ${R}\in (0, \infty)$. By an isometry of $\hh^{n+1}$, we may assume $B^{\hh}_{R}$ is centered at the origin. 
$B^\hh_{R}$, when viewed as a set in $\bb^{n+1}\subset \rr^{n+1}$,  is the Euclidean ball with radius $R_{\rr}:=\sqrt{\frac{1-\arccosh R}{1+\arccosh R}} \in (0, 1)$. 
The principal curvatures of $\p B^\hh_{R}$ are $\coth R$. The unit normal $\bar N$ to $\p B^{\hh}_R$ with respect to $\bar g$ is given by \begin{eqnarray*}
\bar N=\frac{1}{\sinh R}x.
\end{eqnarray*}

As in the Euclidean case, for each constant vector field $a\in \mathbb{R}^{n+1}$, define a corresponding smooth vector field $X_a$ in $\hh^{n+1}$ by
\begin{eqnarray}\label{ConfKillv-h}
X_a= \frac{2}{1-R_{\rr}^2}\left[\<x, a\>x-\frac12(|x|^2+R_\rr^2)a\right].
\end{eqnarray}
Moreover, we define another smooth vector field $Y_a$ in $\hh^{n+1}$ by
\begin{eqnarray}\label{KillY}
Y_a=\frac12(|x|^2+1)a-  \<x, a\>x.
\end{eqnarray}

\begin{prop}\label{lem1h}\
\begin{enumerate}
\item[(i)] $X_a$ is a conformal Killing vector field in $\hh^{n+1}$ with
\begin{eqnarray}\label{confKillhyper1}
\frac12(\bar \n_i (X_a)_j+ \bar \n_j (X_a)_i)=V_{a}\bar g_{ij}, \hbox{ where }V_{a}= \frac{2\<x, a\>}{1-|x|^2}.
\end{eqnarray}
\item[(ii)] $X_a|_{\p B_R^{\hh}}$ is a tangential vector field on $\p B^{\hh}_{R}$. In particular, $$\bar g(X_a,\bar N)=0.$$
\item[(iii)] $Y_a$ is a Killing vector field in $\hh^{n+1}$, i.e.,
\begin{eqnarray}\label{KillY1}
\frac12(\bar \n_i (Y_a)_j+ \bar \n_j (Y_a)_i)=0.
\end{eqnarray}
\end{enumerate}
\begin{remark} $Y_a= \lim _{R_\R\to 1} (R_{\R}-1)X_a$. Though $X_a$ and $Y_a$ look very similar, they are quite different.
$Y_a$ is the Killing vector field induced by the isometry of ``translation'' in $\hh^{n+1}$, while
$X_a$ is a special conformal vector field added by a translation as in the Euclidean case.
For our purpose, $Y_a$ in $\hh^{n+1}$ plays a 
similar role as a constant vector field $a$ in $\rr^{n+1}$.
\end{remark}
\end{prop}
\begin{proof} These are known facts. For the convenience of reader we give a proof.

(i) Recall that $X_a$ is a conformal Killing vector field in  the Euclidean unit ball $\bb^{n+1}$ with respect to the Euclidean metric (Proposition \ref{ConfK}). A well known fact is that a conformal Killing vector field is still a conformal one with respect to a conformal metric, see e.g. \cite{BE}. To be precise,
\begin{eqnarray*}
\frac12(\bar \n_i (X_a)_j+ \bar \n_j (X_a)_i)=\frac{1}{n+1}\div_{\bar g}(X_a)\bar g_{ij},
\end{eqnarray*}
where
\begin{eqnarray*}
\div_{\bar g}(X_a)&=&\div_\delta(X_a)+(n+1) du(X_a)
\\&=&\frac{2}{1-R_\rr^2}(n+1)\<x, a\>+(n+1)\left\<\frac{2x}{1-|x|^2}, \frac{2}{1-R_{\rr}^2}\left[\<x, a\>x-\frac12(|x|^2+R_\rr^2)a\right]\right\>
\\&=&(n+1)\frac{2\<x, a\>}{1-|x|^2}.
\end{eqnarray*}

(ii) This is because $\<X_a, x\>|_{\p B^{n+1}_{R_\rr}}=0$ in  the Euclidean metric and the fact that a conformal transformation preserves the angle.

(iii) As in (i), we know that $Y_a$ is a conformal Killing vector field in $\bb^{n+1}$ with respect to the Euclidean metric. Thus $Y_a$ is again  a conformal Killing one with respect to the conformal metric $\bar g$ with
\begin{eqnarray*}
\frac12(\bar \n_i (Y_a)_j+ \bar \n_j (Y_a)_i)=\frac{1}{n+1}\div_{\bar g}(Y_a)\bar g_{ij},
\end{eqnarray*}
where
\begin{eqnarray*}
\div_{\bar g}(Y_a)&=&\div_\delta(Y_a)+(n+1) du(Y_a)
\\&=& -(n+1)\<x, a\>-(n+1)\left\<\frac{2x}{1-|x|^2}, \<x, a\>x-\frac12(|x|^2+1)a\right\>
\\&=&0.
\end{eqnarray*}
\

\end{proof}


\begin{prop}\label{lem2h} The functions $V_0$ and $V_a$ satisfy
\begin{eqnarray}
&&\bar \n^2 V_0=V_0 \bar g,\label{V0}
\\&&\bar \n^2 V_{a}= V_{a} \bar g.\label{Va}
\end{eqnarray}
\end{prop}

\begin{proof}Identity \eqref{V0} is clear because $V_0=\cosh r$. We verify next \eqref{Va}.
Using the conformal transformation law of the Laplacian,
one can compute directly that 
\begin{eqnarray}\label{deltaV}
&&\bar \Delta V_{a}=e^{-2u}(\Delta_\delta V_{a}+(n-1)du(V_{a}))=(n+1)V_{a}.
\end{eqnarray}
Using \eqref{confKillhyper1} and the commutation formula
$$\bar R_{ijkl}=\bar g(\bar R(\p_i,\p_j)\p_k, \p_l)=\bar g(\bar\n_i\bar \n_j \p_k-\bar\n_j\bar \n_i \p_k, \p_l)$$
and 
$$\bar R_{ijkl}=-(g_{il}g_{jk}-g_{ik} g_{jl}),$$
we compute
\begin{eqnarray*}
(n+1)\bar \n_{i}\bar \n_j V_a&=& \bar \n_i\bar \n_j \bar \n_k (X_a)^k
\\&=& \bar \n_i (\bar \n_k \bar \n_j (X_a)^k+(X_a)^l\bar R_{jkl}^{\ \ \ k})
\\&=&\bar \n_i \bar \n_k (2V_a\delta_j^k- \bar \n^k (X_a)_j)+n \bar \n_i (X_a)_j
\\&=&2 \bar \n_{i}\bar \n_j V_a-\bar \n_i\bar \n_k \bar \n^k (X_a)_j+n \bar \n_i (X_a)_j.
\end{eqnarray*}
Further,
\begin{eqnarray*}
\bar \n_i \bar \n_k \bar \n^k (X_a)_j&=&\bar \n_k \bar \n_i \bar \n^k (X_a)_j-\bar \n^k (X_a)_l\bar R_{ik\ j}^{\ \ l}+\bar \n^l (X_a)_j\bar R_{ikl}^{\ \ \ k}
\\&=&\bar \n_k (\bar \n^k \bar \n_i (X_a)_j+ (X_a)_l \bar R_{i\ \ j}^{\ kl})+\bar \n^k (X_a)_l \bar R_{ik\ j}^{\ \ l}+\bar \n^l (X_a)_j \bar R_{ikl}^{\ \ \ k}
\\&=&\bar \De \bar \n_i (X_a)_j-2\bar \n^k(X_a)_l (\bar g_{ij}\d^l_k- \d_i^l \bar g_{kj})+\bar \n^l (X_a)_j n\bar g_{li}
\\&=&\bar \De \bar \n_i (X_a)_j-2 \bar{{\rm div}}(X_a) \bar g_{ij}+ 2\bar \n_j(X_a)_i+n\bar \n_i (X_a)_j
\end{eqnarray*}
and
\begin{eqnarray}\label{vij}
&&(n-1)\bar \n_{i}\bar \n_j V_a=- \bar \De \bar \n_i (X_a)_j+  2\bar{{\rm div}}(X_a) \bar g_{ij}-2\bar \n_j(X_a)_i.
\end{eqnarray}
Commutating the indices $i$ and $j$ in \eqref{vij}, summing up,  and using \eqref{deltaV} we obtain
\begin{eqnarray*}
2(n-1)\bar \n_{i}\bar \n_j V_a&=&- \bar \De (\bar \n_i (X_a)_j+\bar \n_j (X_a)_i)+ 4\bar{{\rm div}}(X_a) \bar g_{ij}-2(\bar \n_i (X_a)_j+\bar \n_j (X_a)_i)
\\&=&-2\bar \De V_a \bar g_{ij}+  4(n+1)V_a\bar g_{ij}-4V_a \bar g_{ij}
\\&=& 2(n-1)V_a \bar g_{ij}.
\end{eqnarray*}
Identity \eqref{Va} follows.
\end{proof}
\begin{remark}\label{rem-classi}
We remark that in $\hh^{n+1}$, the vector space $\{V\in C^2(\hh^{n+1}): \bar \n^2 V=V\bar g\}$ is spanned by $V_0$ and $V_a, a\in \rr^{n+1}$. Thus it has dimension $n+2$. 
\end{remark}

Note that the vector field $a$ is not a constant (or parallel) with respect to the hyperbolic metric.  In the following we derive formulae of covariant derivatives of several functions and vector fields associated with $a$. We will frequently use \eqref{Poincare} and \eqref{hypfunc}.
\begin{prop}For any tangential vector field $Z$ on $\hh^{n+1}$, 
\begin{eqnarray}
&&\bar \n_{Z}a=e^{-u}\left[\bar g(x, Z)a+\bar g(x, a)Z- \bar g(Z, a)x\right], \label{eq1h}
\\&&\bar \n_{Z}(e^{-u}a)=e^{-u}\left[\bar g(x, e^{-u}a)Z- \bar g(Z, e^{-u}a)x\right].\label{eq2h}
\\&&\bar \n_{Z} V_0= \bar g(x, Z),\label{eq3h'}
\\&&\bar \n_{Z}V_a=\bar g(Z, e^{-u}a)+ e^{-u}\bar g(x, e^{-u}a)\bar g(Z, x),\label{eq3h}
\\&&\bar \n_{Z} Y_a=e^{-u} \bar g(x, Z) a- e^{-u} \bar g(Z, a)x.\label{eq4h}
\\&&\bar \n_{Z} X_a=-\cosh R[e^{-u} \bar g(x, Z) a- e^{-u} \bar g(Z, a)x] +e^{-u}\bar g(x, a)Z.\label{eq5h}
\end{eqnarray}
\end{prop}
\begin{proof}
Let $\{E_i\}_{i=1}^{n+1}$ be the coordinate unit vector in $\rr^{n+1}$. Let $Z=Z^iE_i$ and $a=a^i E_i$.
Under the conformal transformation,
\begin{eqnarray*}
\bar \n_{E_i} E_j&=&E_i(u) E_j+E_j(u) E_i- \<E_k(u), E_k\>\delta_{ij}
\\&=&\frac{2}{1-|x|^2}(x_i E_j+ x_j E_i -x\delta_{ij}).
\end{eqnarray*}
It follows that 
\begin{eqnarray*}
\bar \n_{Z}a&=&Z^i a^j\bar \n_{E_i} E_j
\\&=&Z^i a^j\frac{2}{1-|x|^2}(x_i E_j+ x_j E_i -x\delta_{ij})
\\&=&e^{-u}\left[\bar g(x, Z)a+\bar g(x, a)Z- \bar g(Z ,a)x\right],
\end{eqnarray*}
where we have used $e^{-u}=\frac{1-|x|^2}{2}$ and $\bar g=e^{2u}\<\cdot, \cdot\>$.  It is easy to check 
\begin{eqnarray}\label{eq-pf}
\bar \n_Z (e^{-u})=-e^{-u}Z(u) = -e^{-2u}\bar g(x, Z).
\end{eqnarray}
Equation \eqref{eq2h} follows then from \eqref{eq1h} and  \eqref{eq-pf}. Equation \eqref{eq3h'} follow easily from $V_0=\cosh r$ and $x=\sinh r\p_r$.

We rewrite $V_a$ as \begin{eqnarray}\label{eq-pf2}
V_a=\frac{2\<x, a\>}{1-|x|^2}=\bar g(x, e^{-u}a). 
\end{eqnarray}
We compute $V_a$ using  \eqref{eq-pf2}. 
Using \eqref{rad-conf} and \eqref{eq2h}, we get
\begin{eqnarray*}
\bar \n_{Z}  V_a&=&\bar g(\bar \n_{Z}  x,  e^{-u}a)+\bar g(x, \bar \n_{Z}  (e^{-u}a))
\\&=&V_0 e^{-u}\bar g(Z, a)+e^{-2u}\left[\bar g(x, a)\bar g(x, Z)- \bar g(Z, a)\bar g(x, x)\right]
\\&=&e^{-u}\bar g(Z, a)+ e^{-2u}\bar g(x, a)\bar g(Z, x).
\end{eqnarray*}
This is \eqref{eq3h}. In the last equality, we have used $V_0- e^{-u}\bar g(x, x)=\cosh r-\frac{1}{1+\cosh r}\sinh^2 r=1.$ 

Recall $Y_a=\frac12(|x|^2+1)a-\<x, a\>x$.
Using \eqref{eq1h} and \eqref{rad-conf}, we have
\begin{eqnarray*}
\bar \n_{Z} Y_a&=& \<x, Z\>a +\frac12(|x|^2+1)\bar \n_Z a- \<Z, a\>x-\<x, a\>\bar \n_Z x 
\\&=& e^{-2u}\bar g(x, Z) a+\frac12(|x|^2+1)e^{-u}\left[\bar g(x, Z)a+\bar g(x, a)Z- \bar g(Z, a)x\right]
\\&&- e^{-2u} \bar g(Z, a)x- e^{-2u}\bar g(x, a)V_0 Z
\\&=&e^{-u} \bar g(x, Z) a- e^{-u} \bar g(Z, a)x.
\end{eqnarray*}

The proof of equation \eqref{eq5h} is similar to that of \eqref{eq4h}.
\end{proof}

Let $x: M\to B^\hh_R$ be an isometrically immersed hypersurface which intersects $\p B^\hh_R$ at a constant angle $\th$.  As in the Euclidean space, by using  properties of $X_a$ and $Y_a$ in Proposition \ref{lem1h} and the fact that $\p B^\hh_R$ is umbilical in $\hh^{n+1}$, we have the following Minkowski type formula.
\begin{prop}[Minkowski formula]\label{Minksf}  Let $x: M\to B^\hh_R$ be an  isometric immersion into the hyperbolic ball $B^\hh_R$,  whose boundary $\p M$ intersects $\p B^\hh_R$ at a constant angle $\th \in (0, \pi)$. Let $a\in \rr^{n+1}$ be a constant vector field and $X_a$, $Y_a$ are defined by \eqref{ConfKillv-h} and \eqref{KillY}. Then
\begin{eqnarray}\label{Mink1h}
\int_M n (V_a +\sinh R\cos \th \, \bar g(Y_a, \nu)) dA=\int_M H\bar g(X_a, \nu) dA.
\end{eqnarray}
\end{prop}
\begin{proof} Similarly as in the proof of Proposition \ref{MinkEucl}, by using the two properties of $X_a$ in Proposition \ref{lem1h}, we get
 \begin{eqnarray}\label{min1'}
&&\int_M n V_a - H \bar g(X_a,\nu) dA
=\int_M \div_M (X_a^T) dA
\\&=& \int_{\p M} \bar g(X_a^T, \mu) ds=-\frac{2R_\rr^2}{1-R_\rr^2}\cos \th\int_{\p M}\bar g(a, \bar \nu)ds.\nonumber
\end{eqnarray}
Set $$Z_a=\bar g(\nu, e^{-u}a)x- \bar g(x, \nu)(e^{-u}a).$$
We claim that 
\begin{eqnarray}\label{Za}
\div_M Z_a= n\bar g(Y_a, \nu).
\end{eqnarray}
Indeed, by a direct computation we have 
\begin{eqnarray*}
&&\div_M[\bar g(\nu, e^{-u}a)x]= h(a^T, x^T)-e^{-u} \bar g(x^T, e^{-u}a)\bar g(x,\nu)+\bar g(\nu, e^{-u}a)(n V_0 - H\bar g(x, \nu)),
\end{eqnarray*}
and
\begin{eqnarray*}
&&\div_M[\bar g(x, \nu)(e^{-u}a)]
=h(x^T, a^T)+\bar g(x, \nu)[e^{-u}(n\bar g(x, e^{-u}a)-\bar g(x^T, e^{-u}a)) - H\bar g(\nu, e^{-u}a)].
\end{eqnarray*}
It follows that
\begin{eqnarray*}
\div_M Z_a&=& nV_0\bar g(\nu, e^{-u}a)- ne^{-u} \bar g(x, e^{-u}a)\bar g(x, \nu)
\\&=&n\bar g(\nu, \frac12(|x|^2+1)a- \<x, a\>x)
\\&=& n\bar g(Y_a, \nu),
\end{eqnarray*}
where we have used $V_0=\frac{1+|x|^2}{1-|x|^2}$, $e^{-u}=\frac{1-|x|^2}{2}$ and $\bar g=e^{2u}\<\cdot, \cdot\>$. Thus we proved the claim.

Integrating \eqref{Za} over $M$ and using integration by parts, we have
\begin{eqnarray}\label{eqmin1}
\int_M n\bar g(Y_a, \nu) dA= \int_{\p M} \bar g(Z_a, \mu) ds.
\end{eqnarray}
Using \eqref{mu0} and \eqref{nu0}, It is easy to check that 
\begin{eqnarray}\label{eqmin2}
\bar g(Z_a, \mu)|_{\p M}= R_\rr \bar g(\bar \nu, a).
\end{eqnarray}
The Minwokski formual \eqref{Mink1h} follows from \eqref{min1'}, \eqref{eqmin1} and \eqref{eqmin2}.
\end{proof}

\begin{prop}\label{bdrylemma} Along $\p M$, we have
\begin{eqnarray} 
\bar \n_{\mu}(V_a +\sinh R\cos \th \, \bar g(Y_a, \nu)) &=& q (V_a +\sinh R\cos \th \, \bar g(Y_a, \nu)), \label{bdry1-h}
\\\bar \n_{\mu} \bar g(X_a, \nu)&=&q \, \bar g(X_a, \nu), \label{bdry2-h}
\end{eqnarray}
where
\begin{eqnarray}\label{qq2}
&&q=\frac{1}{\sin \th}\coth R+ \cot \th \, h(\mu, \mu).
\end{eqnarray}

\end{prop}
\begin{proof}
In this proof we always take value along $\p M$ and use \eqref{mu0} and \eqref{nu0}. 
First, note that
\begin{eqnarray*}
&&\bar g(Y_a, x)= e^{2u}\<Y_a, x\>=e^{2u}\frac12(|x|^2-1)\<x, a\>=e^{-u}\bar g(x, a)=V_a.
\end{eqnarray*}
Thus we have
\begin{eqnarray}\label{bdry-1hh}
V_a+ \sinh R\cos \th \, \bar g(Y_a, \nu)&= &\bar g(Y_a, x+ \sinh R\cos \th\,\nu)
\\&=&\bar g(Y_a, \sinh R \, \bar N+ \sinh R\cos \th\,(-\cos\th \bar N+\sin \th \bar \nu))\nonumber
\\&=&\sinh R\sin \th \bar g(Y_a, \mu).\nonumber
\end{eqnarray}
 By \eqref{eq3h} and \eqref{eq4h}, we compute
\begin{eqnarray*}
&&\bar \n_\mu(V_a+ \sinh R\cos \th \, \bar g(Y_a, \nu))
\\&= &e^{-u}\bar g(\mu, a)+ e^{-2u}\bar g(x, a)\bar g(\mu, x)\nonumber
\\&&+\sinh R\cos \th \, \bar g(Y_a, h(\mu, \mu)\mu)+\sinh R\cos\th \, e^{-u} [\bar g(x,\mu)\bar g(\nu, a)-\bar g(\mu, a)\bar g(x, \nu)].\nonumber\end{eqnarray*}
Using $\nu=-\frac{1}{\cos \th}\bar N+\tan \th\, \mu$ and $x=\sinh R \, \bar N$, we obtain
\begin{eqnarray*}
&&\sinh R\cos\th \, e^{-u} [\bar g(x,\mu)\bar g(\nu, a)-\bar g(\mu, a)\bar g(x, \nu)]
\\&=&\sinh^2 R \, e^{-u}\bar g(\mu, a)-e^{-u} \bar g(x, \mu) \, \bar g(x, a).\nonumber
\end{eqnarray*}
Therefore, we have
\begin{eqnarray}\label{bdry-2hh}
&&\bar \n_\mu(V_a+ \sinh R\cos \th \, \bar g(Y_a, \nu))
\\&= &e^{-u}\cosh^2 R\,\bar g(\mu, a)+  (e^{-2u}-e^{-u})\bar g(x, a)\bar g(x, \mu)+\sinh R\cos \th \, \bar g(Y_a, h(\mu, \mu)\mu)\nonumber
\\&=&\cosh R\, \bar g\left(\frac12(|x|^2+1)a-\<x, a\>x, \mu\right)+\sinh R\cos \th \, \bar g(Y_a, h(\mu, \mu)\mu)
\nonumber
\\&=&(\cosh R+ \sinh R\cos \th \,h(\mu, \mu))  \bar g(Y_a, \mu).\nonumber\end{eqnarray}
The first formula \eqref{bdry1-h} follows from \eqref{bdry-1hh} and \eqref{bdry-2hh}.

Next, using $\bar N=\sin \th \,\mu-\cos\th\,\nu$, we get
\begin{eqnarray}\label{bdry-3hh}
\bar g(X_a, \nu)&=&\frac{2}{1-R_\rr^2}\left[e^{-2u}\bar g(x, a)\bar g(x,\nu)-\frac12(|x|^2+R_\rr^2)\bar g(a, \nu)\right]
\\&=&\frac{2}{1-R_\rr^2}\left[-\cos \th \, R_\rr^2 \bar g(\sin \th \,\mu-\cos\th\,\nu, a)+ R_\rr^2\bar g(a, \nu)\right]\nonumber
\\&=&- \frac{2R_\rr^2}{1-R_\rr^2}\sin \th\,\bar g(\cos \th\, \mu+\sin \th\,\nu, a).\nonumber
\end{eqnarray}
 Since $\nu=-\frac{1}{\cos \th}\bar N+\tan \th\, \mu$ and $X_a\perp \bar N$, we have
\begin{eqnarray}\label{bdry-4hh}
\bar g(X_a, \nu) =\tan \th\, \bar g(X_a, \mu).
\end{eqnarray}
In view of  \eqref{eq5h} and Proposition \ref{lemma1}, we have
\begin{eqnarray}\label{bdry-5hh}
\bar \n_\mu \bar g(X_a, \nu)&=&-\cosh R[e^{-u} \bar g(x, \mu) \bar g(a, \nu)- e^{-u} \bar g(\mu, a)\bar g(x,\nu)]
\\&& +e^{-u}\bar g(x, a)\bar g(\mu, \nu)+\bar g(X_a, h(\mu,\mu)\mu)\nonumber
\\&=&-\cosh R \, R_\rr[\sin \th\, \bar g(a, \nu)+ \cos\th \, \bar g(\mu, a)]+h(\mu,\mu)\bar g(X_a, \mu)\nonumber
\\&=&-\coth R \,  \frac{2R_\rr^2}{1-R_\rr^2}\bar g(\cos \th\, \mu+\sin \th\,\nu, a) +h(\mu,\mu)\bar g(X_a, \mu)\nonumber
\\&=& \nonumber \frac 1 {\sin \theta} \coth R \, \bar g (X_a, \nu) +h(\mu,\mu)\bar g(X_a, \mu),
\end{eqnarray}
where in the last equality, we have used \eqref{bdry-3hh}.
The second assertion \eqref{bdry2-h}  follows. 
The proof is completed.
\end{proof}

\begin{prop}\label{prop4.6} Let $x: M\to (\bb^{n+1}, \bar g)$ be an isometric immersion into the hyperbolic  Poincar\'e ball. Let $a$ be a constant vector field in $\rr^{n+1}$. The following identities hold along $M$:
\begin{eqnarray}
&& \De V_0= n V_0-H\bar g(x, \nu),\label{eq-x0-h}
\\&&\De V_a=n V_a-H \bar \n_\nu V_a,\label{eq-x-h}
\\&&\De\bar g(x, \nu)=HV_0+ \bar g(x, \n H)-|h|^2\bar g(x, \nu), \label{eq-xnu-h}
\\&&\De \bar g(X_a, \nu)=HV_a+ \bar g(X_a, \n H) -|h|^2\bar g(X_a, \nu) -n \bar \n_{\nu} V_a +n\bar g(\nu, X_a), \label{eq-Xnu-h}
\\&&\De \bar g(Y_a, \nu)=  -|h|^2\bar g(Y_a, \nu) +\bar g(Y_a, \n H)  +n\bar g(Y_a, \nu).\label{eq-Ynu-h}
\end{eqnarray}
\end{prop}
\begin{proof} \eqref{eq-x0-h} and \eqref{eq-x-h} follow from \eqref{V0} and \eqref{Va} respectively and the Weingarten formula.

We prove next \eqref{eq-Xnu-h}. Choose an local normal frame $\{e_\a\}_{\a=1}^{n}$ at a given point $p$, i.e., $\n_{e_\a}e_{\b}|_p=0$. Denote by $\mathcal{W}: TM\to TM$ the Weingarten map.
We will frequently use the conformal property \eqref{confKillhyper1} of $X_a$. We compute at $p$,
\begin{eqnarray*}
e_\a \bar g(X_a, \nu)&=&\bar g(X_a, \w(e_\a))+\bar g(\bar \n_{e_\a} X_a, \nu)\\&=& \bar g(X_a, \w(e_\a))-\bar g(\bar \n_{\nu} X_a, e_\a),
\end{eqnarray*}
and
\begin{eqnarray*}
\De \bar g(X_a, \nu)&=&e_\a e_\a  \bar g(X_a, \nu)
\\&=&\bar g(\bar \n_{e_\a} X_a, \w(e_\a))+ \bar g(X_a, \bar \n_{e_\a} (\w(e_\a)))\\&&-\bar g(\bar \n_{e_\a} (\bar \n_{\nu} X_a), e_\a)-\bar g(\bar \n_{\nu} X_a, -H\nu)
\\&=&h_{\a\b}\bar g(\bar \n_{e_\a} X_a, e_\b)+ \bar g(X_a, (\n_{e_\a} \w)(e_\a) -h(e_\a, \w(e_\a))\nu)
\\&&-\bar g(\bar \n_{e_\a} \bar \n_{\nu} X_a, e_\a)+ HV_a
\\&=&HV_a+ \bar g(X_a, \n H) -|h|^2\bar g(X_a, \nu)-\bar g(\bar \n_{e_\a} \bar \n_{\nu} X_a, e_\a).
\end{eqnarray*}
Using the definition of Riemannian curvature tensor and the fact that the ambient space has curvature $-1$, we get
\begin{eqnarray*}
&&-\bar g(\bar \n_{e_\a} \bar \n_{\nu} X_a, e_\a)\\&=&-\bar g(\bar \n_{\nu} \bar \n_{e_\a} X_a, e_\a)+\bar g(\bar \n_{[\nu, e_\a]} X_a, e_\a)+ \bar g(\bar R(\nu, e_\a)X_a, e_\a)
\\&=&-\bar \n_{\nu}\bar g(\bar \n_{e_\a} X_a, e_\a)+\bar g(\bar \n_{e_\a} X_a,  \bar \n_{\nu} e_\a)+\bar g(\bar \n_{[\nu, e_\a]} X_a, e_\a)+ n\bar g(\nu, X_a)
\\&=&-n \bar \n_{\nu} V_a+\bar g(\bar \n_{e_\a} X_a,  \bar \n_{e_\a}\nu+[\nu, e_\a])+\bar g(\bar \n_{[\nu, e_\a]} X_a, e_\a)+ n\bar g(\nu, X_a)
\\&=&-n \bar \n_{\nu} V_a+HV_a  +\bar g([\nu, e_\a], e_\a)V_a+ n\bar g(\nu, X_a).
\end{eqnarray*}
Furthermore the Koszul formula gives
\begin{eqnarray*}
2\bar g(\bar \n_{e_\a} \nu, e_\a)=-\bar g([\nu, e_\a], e_\a)- \bar g([e_\a, e_\a], \nu)+\bar g([e_\a, \nu], e_\a),
\end{eqnarray*}
which implies
\begin{eqnarray*}
\bar g([\nu, e_\a], e_\a)=-H.
\end{eqnarray*}
Combining the above, we get \eqref{eq-Xnu-h}.

By taking account of the fact that $x$ has the conformal Killing property   \eqref{rad-conf} and $Y_a$ has the Killing property \eqref{KillY1}, \eqref{eq-xnu-h} and \eqref{eq-Ynu-h} follow similarly as \eqref{eq-Xnu-h}.
\end{proof}

\subsection{Uniqueness of stable capillary hypersurfaces in a hyperbolic ball}\

\begin{theorem}\label{thm-h} Assume $x: M\to B_R^{\hh}\subset(\bb^{n+1}, \bar g)$ is an immersed stable capillary in the ball $B_R^{\hh}$ with constant mean curvature $H \ge 0$ and constant contact angle $\th\in (0, \pi)$. Then $x$ is totally umbilical.
\end{theorem}

\begin{proof}
The stability inequality \eqref{stab-ineq}  reduces to
\begin{eqnarray}\label{stab-eq1-h}
-\int_M \varphi (\Delta \varphi+|h|^2\varphi-n\varphi)-\int_{\p M} (\n_\mu \varphi-q\, \varphi)\varphi\ge 0
\end{eqnarray}
for all function $\varphi\in \mathcal{F}$, where $q$ is given by \eqref{qq2} since $\p B_R^{\hh}$ has constant principal curvature $\coth R$.

For each constant vector field $a\in \rr^{n+1}$, we consider a test function
 $$\varphi_a=n (V_a+\sinh R\cos \th \bar g(Y_a, \nu))-H\bar g(X_a,\nu)$$ along $M$.
The Minkowski type formula \eqref {Mink1h} tells us that $\int_M \varphi_a dA=0$. Therefore, $\varphi_a \in \mathcal{F}$ and
is an admissible function for testing stability.
Using \eqref{eq-x-h}, \eqref{eq-Xnu-h} and \eqref{eq-Ynu-h}, noting that $H$ is a constant,  we easily see that 
\begin{eqnarray}\label{eq-phi-h}
&&\Delta \varphi_a+|h|^2\varphi_a-n\varphi_a=(n|h|^2-H^2)V_a.
\end{eqnarray}
From \eqref{bdry1-h} and \eqref{bdry2-h}, we know  \begin{eqnarray}\label{eq-phi1-h}
\n_\mu \varphi_a-q\varphi_a=0.
\end{eqnarray}

Inserting \eqref{eq-phi-h} and \eqref{eq-phi1-h} into the stability condition \eqref{stab-eq1-h}, we get for any $a\in \mathbb{R}^{n+1}$, 
\begin{eqnarray}\label{stab-eq2'}
&&\int_M \left[n(V_a+\sinh R\cos \th \bar g(Y_a, \nu))-H\bar g(X_a,\nu)\right]V_a(n|h|^2-H^2) \, dA\le 0. 
\end{eqnarray}
We take $a$ to be the $n+1$ coordinate vectors $\{E_i\}_{i=1}^{n+1}$ in $\R^{n+1}$.
Noticing  that $V_a=\frac{2\<x, a\>}{1-|x|^2}$, $X_a=\frac{2}{1-R_\rr^2}\left(\<x, a\>x-\frac12(|x|^2+R_\rr^2)a\right)$ and $Y_a=\frac12(|x|^2+1)a- \<x, a\>x$,
we have \begin{eqnarray*}
&&\sum_{a=1}^{n+1}V_a^2=\frac{4|x|^2}{(1-|x|^2)^2}=\bar g(x, x),
\\&&\sum_{a=1}^{n+1} V_aX_a= \frac{2}{1-R_\rr^2}\frac{|x|^2-R_\rr^2}{1-|x|^2}x=(V_0-\cosh R)x,
\\&&\sum_{a=1}^{n+1} V_a Y_a=x.
\end{eqnarray*}
Therefore, by summing \eqref{stab-eq2'} for all $a$, we get
\begin{eqnarray}\label{stab3''}
&&\int_\S \left[ n(\bar g(x, x)+\sinh R\cos \th \bar g(x,\nu))-(V_0-\cosh R)H\bar g(x,\nu)\right](n|h|^2-H^2)\le 0.
\end{eqnarray}

As in the Euclidean case, we introduce an auxiliary function
\begin{eqnarray*}
\Phi= \left(V_0-\cosh R\right)H-n(\bar g(x,\nu)+\cos \th \sinh R).
\end{eqnarray*}
From \eqref{eq-x0-h} and \eqref{eq-xnu-h}, we get
\begin{eqnarray}\label{stab4'}
\De \Phi=(n|h|^2-H^2)\bar g(x,\nu).
\end{eqnarray}
Note that  $\Phi|_{\p M}=0$. Thus we have
\begin{eqnarray*}
\int_M \De \frac12\Phi^2 dA=\int_{\p M} \Phi \n_\mu\Phi ds=0.
\end{eqnarray*}
Adding this to \eqref{stab3''}, using \eqref{stab4'}, we have
\begin{eqnarray*}
0&\ge&
\int_M \left(n \bar g(x, x)-\left(\cosh r-\cosh R\right)H\bar g(x,\nu)\right)(n|h|^2-H^2)+ \De \frac12\Phi^2
\\&=&\int_M n\bar g(x^T, x^T)(n|h|^2-H^2)+|\n \Phi|^2
\\&\ge &0.
\end{eqnarray*}
The same argument as before yields the umbilicy of the immersion $x$. This implies $x: M\to  B_R^{\hh}$ is either part of a totally geodesic hypersurface  or part of a geodesic ball.
The proof is completed.
\end{proof}

\subsection{The case $\ss^{n+1}$}\

In this subsection, we sketch  the necessary  modifications in the case that the ambient space is the spherical space form $\ss^{n+1}$. 
We use the model $$(\rr^{n+1}, \bar g_{\ss}=e^{2u} \delta) \quad \hbox{ with  }u(x)=\frac{4}{(1+|x|^2)^2},$$to represent $\ss^{n+1}\setminus\{\mathcal{S}\}$, the unit sphere without the south pole.
 Let $B_R^{\ss}$ be a ball in $\ss^{n+1}$ with radius $R\in (0, \pi)$ centered at the north pole. The corresponding $R_\rr=\sqrt{\frac{1-\cos R}{1+\cos R}}\in (0, \infty)$.
The crucial conformal Killing vector field $X_a$ and the Killing vector field $Y_a$ in this case are
\begin{eqnarray}\label{confKillsph}
X_a=\frac{2}{1+R_{\rr}^2}\left[\<x, a\>x-\frac12(|x|^2+R_\rr^2)a\right],
\end{eqnarray}
\begin{eqnarray}\label{Killsph}
Y_a=\frac12(1-|x|^2)a+\<x, a\>x.
\end{eqnarray}
The crucial functions $V_0$ and $V_a$ in this case are
$$V_0= \cos r=\frac{1-|x|^2}{1+|x|^2}, \quad V_a=\frac{2\<x, a\>}{1+|x|^2}.$$ 
Similarly as the hyperbolic case, these $(n+2)$ functions span the vector space  $$\{V\in C^2(\ss^{n+1}\setminus\{\mathcal{S}\}): \bar \n^2 V= - V\bar g\}.$$
Using $X_a$, $Y_a$, $V_0$ and $V_a$, the proof goes through parallel to the hyperbolic case.
The method works for balls with any radius $R\in (0, \pi)$. Compare to the hyperbolic case,  in this case $V_0=\cos r$ can be negative when $R\in (\frac{\pi}{2}, \pi)$. Nevertheless, by going through the proof, we see this does not affect the issue on stability. We leave the details to the interested reader.
\qed


\subsection{Exterior problem}\

To end this section, we give a sketch of proof for the exterior problem, Theorem \ref{thm0.4}. We take  the hyperbolic case as an example.
\begin{theorem}\label{thm-h-ext} Assume $x: M\to \hh^{n+1}\setminus B_R^{\hh}$ is a compact immersed stable capillary hypersurface outside the hyperbolic ball $B_R^{\hh}$ with constant mean curvature $H \ge 0$ and constant contact angle $\th\in (0, \pi)$. Then $x$ is totally umbilical.
\end{theorem}
\begin{proof} 
In this case, the differences occur that $x=-\sinh R\, \bar N$ and
the term $q$ in the stability inequality \eqref{stab-ineq} is given by
$$q=-\frac{1}{\sin\th}\coth R+\cot \th \, h(\mu,\mu).$$
By checking the proof of Proposition \ref{Minksf}, we see the Minkowski formula is
$$\int_M n (V_a -\sinh R\cos \th \, \bar g(Y_a, \nu)) dA=\int_M H\bar g(X_a, \nu) dA.$$
We take the test function to be
\begin{eqnarray}\label{testf}\varphi_a=  n (V_a -\sinh R\cos \th \, \bar g(Y_a, \nu))- H\bar g(X_a, \nu).\end{eqnarray} Then $\int_M \varphi_a dA=0$. Also, by checking the proof of Proposition \ref{bdrylemma}, we see that
$\bar \n_{\mu}\varphi_a= q\varphi_a$ along $\p M$. From Proposition \ref{prop4.6}, $\varphi_a$ in \eqref{testf} still satisfies  \eqref{eq-phi-h}. Then the proof is exactly the same as the interior problem, Theorem \ref{thm-h}.
\end{proof}

\medskip

\section{Heintze-Karcher-Ros type Inequality and Alexandrov Theorem}\label{sec5}

Let $K=0 \hbox{ or } \pm 1$.
Denote by $\bar \mm^{n+1}(K)$  the space form with sectional curvature $K$. As in previous section, we use the Poincar\'e ball model $(\bb^{n+1}, \bar g_{\hh})$ for $\bar \mm^{n+1}(-1)$ and the model $(\rr^{n+1}, \bar g_\ss)$ for $\bar \mm^{n+1}(1)$.

In this section we consider an {\it isometric embedding} $x: M\to \bar \mm^{n+1}(K)$ into a ball $B$ in a space form with free boundary, \i.e., $\theta=\pi/2$.
To unify the notation, we use $B$ to mean the unit ball $\bb^{n+1}$ in the Euclidean case, the ball $B_R^{\hh}$ with radius $R$ $(R\in (0, \infty))$ in the hyperbolic case and the ball $B_R^\ss$ with radius $R$ $(R\in (0, \pi))$ in the spherical case. We denote $\S=x(M)$.
Let $B$ be decomposed by $\S$ into two connected components. We choose one and denote it by $\O$.
Denote by $T$ the part of $\p \O$ lying on $\p B$. Thus, $\p \O= \S\cup T$.

We also unify the following notations:
\begin{equation*}
V_0=\left\{
\begin{array}{lll}1, &K=0,
\\ \cosh r, &K=-1,
\\ \cos r, &K=1,
\end{array}
\right.
\end{equation*}
and
\begin{equation*}
V_a=\left\{
\begin{array}{lll}\<x, a\>, &K=0,
\\ \frac{2\<x, a\>}{1-|x|^2}, &K=-1,
\\ \frac{2\<x, a\>}{1+|x|^2}, &K=1.
\end{array}
\right.
\end{equation*}
and $X_a$ is conformal vector field in \eqref{ConfKillv}, \eqref{ConfKillv-h} and \eqref{confKillsph} in each case respectively.

We first prove other Minkowski type formulae.

\begin{prop}
Let $x: M\to \bar \mm^{n+1}(K)$ be an embedded smooth hypersurface into $B$ which  meets $B$ orthogonally. Let $\s_{k}, k=1,\cdots, n$ be the $k$-th mean curvatures, i.e., the elementary symmetric functions acting on the principal curvatures. Then
\begin{eqnarray}
\int_\O V_a d\O&=&\frac{1}{n+1}\int_\S \bar g(X_a, \nu)dA. \label{Mink1'}\\
\int_\S V_a \s_{k-1} dA&=&\frac k{n+1-k}  \int_\S \s_k \bar g(X_a, \nu)dA, \quad \forall k=1,\cdots, n.\label{Mink2'}
\end{eqnarray}
 \end{prop}
  \begin{remark}\label{remMink} Formula \eqref{Mink2'} is still true if $x$ is only an immersion.
 \end{remark}
 
\begin{proof} 
Due to the perpendicularity condition, $\mu=\bar N$.
Since $X_a\perp \bar N$  along $\p B$, we see $X_a\perp \mu$  along $\p \S$. From the conformal property, we have $$\div_{\bar g} X_a=(n+1)V_a.$$ Integrating it over   $\O$ and using Stokes' theorem, we have
\begin{eqnarray*}
(n+1)\int_\O V_a d\O=\int_\O \div_{\bar g} X_a d\O=\int_\S \bar g(X_a, \nu)dA +\int_{T} \bar g(X_a, \bar N) dA=\int_\S \bar g(X_a, \nu)dA.
\end{eqnarray*}
This is \eqref{Mink1'}.
Denote by $X_a^T$ the tangential projection of $X_a$ on ${\S}$. From above we know that $X_a^T\perp \mu$  along $\p \S$.
Let $\{e_\a\}_{\a=1}^{n}$ be an orthonormal frame on ${\S}$. From the conformal property, 
we have that \begin{eqnarray}\label{confKill2'}
\frac12\left[\n_\a (X_a^T)_\b+\n_\b (X_a^T)_\a\right]=V_a g_{\a\b}-h_{\a\b}\bar g(X_a,\nu).
\end{eqnarray}
(cf. \eqref{confKill2}.)
Denote $T_{k-1}(h)=\frac{\p \s_k}{\p h}$ the Newton transformation.
Multiplying $T_{k-1}^{\a\b}(h)$ to \eqref{confKill2'} and integrating by parts on $\S$, we get
\begin{eqnarray*}
&&\int_\S (n+1-k) V_a\s_{k-1}(h)-k\s_k(h)\bar g(X_a,\nu) dA
\\&=&\int_\S T_{k-1}^{\a\b}(h)\n_\a (X_a^T)_\b dA=\int_{\p\S} T_{k-1}^{\a\b}(h) \bar g(X_a^T, e_\b)\bar g(\mu, e_\a) ds
\\&= &\int_{\p\S} T_{k-1}(X_a^T, \mu) ds=0. 
\end{eqnarray*}
In the last equality, we have used Proposition \ref{lemma1} and the fact that $X_a^T\perp \mu$ along $\p\S$.
In fact, since $\mu$ is a principal direction of $h$, it is also a  principal direction of the Newton tensor $T_{k-1}$ of $h$, which implies that 
 $T_{k-1}(X_a^T, \mu)=0$.
 The proof is completed.
\end{proof}

Next we prove a Heintze-Karcher-Ros type inequality.
In order to prove the  Heintze-Karcher-Ros type inequality, we need a generalized Reilly formula, which has been proved  by  Qiu-Xia, Li-Xia \cite{QX, LX, LX2}.
\begin{theorem}[\cite{QX, LX}]\label{QLX}
Let $\O$ be a bounded domain in a Riemannian manifold $(\bar M, \bar g)$  
with piecewise smooth boundary $\p\O$. Assume that $\p \Omega$ is decomposed into two smooth pieces $\p_1 \Omega$ and $\p_2 \Omega$ with a common boundary $\Gamma$. 
Let $V$ be a non-negative smooth function  on $\bar \O$ such that $\frac{\bar \n^2 V}{V}$ is continuous up to $\p\O$.
Then for any function $f\in C^{\infty}(\bar \O\setminus \G)$, 
we have 
\begin{eqnarray}\label{qx}
&&\int_\O V\left(\left(\ode f-\frac{\bar \De V}{V}f\right)^2-\left|\on^2 f-\frac{\bar \n^2 V}{V}f\right|^2\right)d\O\\
&=&\int_\O \left(\ode V\bar g- \on^2 V+V{\rm Ric}\right)\left(\bar \n f-\frac{\bar \n V}{V}f, \bar \n f-\frac{\bar \n V}{V}f\right)d\O\nonumber
 \\&&+\int_{\p\O} V\left(f_\nu-\frac{V_\nu}{V}f\right)\left(\De f-\frac{\De V}{V}f\right)dA - \int_{\p\O} Vg\left(\n \left(f_\nu-\frac{V_\nu}{V}f\right), \n f-\frac{\n V}{V}f\right)dA 
 \nonumber\\&&+\int_{\p\O} VH\left(f_\nu-\frac{V_\nu}{V}f\right)^2 +\left(h-\frac{V_{\nu}}{V}g\right)\left(\n f-\frac{\n V}{V}f, \n f-\frac{\n V}{V}f\right)dA.\nonumber
\end{eqnarray}
\end{theorem}

\begin{remark} The formula \eqref{qx} here is a bit different with that in \cite{LX}. We do not do integration by parts on $\p\O$ in the last step of the proof as \cite{QX, LX}.
\end{remark}


\begin{theorem}\label{HK} Let $x: M\to \bar \mm^{n+1}(K)$ be an embedded smooth hypersurface into $B$ with $\p\Sigma \subset \p B$. 
Assume $\S$ lies in a half ball $$B_{a+}=\{V_a\ge 0\}=\{\<x, a\>\ge 0\}.$$ If $\S$ has positive mean curvature, then
\begin{eqnarray}\label{HKR}
\int_{\S} \frac{V_a}{H} dA\ge \frac{n+1}{n}  \int_\O V_a d\O.
\end{eqnarray}
Moreover, equality in \eqref{HKR} holds if and only if $\S$ is a spherical cap which  meets $\p B$ orthogonally.
\end{theorem}
\begin{proof}Recall $\p \O=\S\cup T$, where $T$ is the boundary part lying in $\p B$. See Figure 1.
Let $f$ be a solution of the mixed boundary value problem
\begin{equation}\label{mixed2}
\left\{
\begin{array}{rccl}
\ode f+K(n+1)f&=&1&
\hbox{ in } \Omega,\\
f&=& 0&\hbox{ on } \S,\\
V_a f_{\bar N}-f(V_a)_{\bar N}&=& 0&\hbox{ on } T. 
\end{array}
\right.
\end{equation}
Since the existence of \eqref{mixed2} has its own interest, we give a proof in Appendix \ref{App}. From the Appendix we have $f\in W^{1,2}(\O)$ satisfying \eqref{mixed2} in the weak sense, i.e., $f=0$ on $\S$ and
\begin{eqnarray}\label{mixed-weak}
\int_\O [\bar g(\bar \n f, \bar \n\varphi) - K(n+1)f\varphi +\varphi] \, d\O=\int_T f \varphi dA, 
\end{eqnarray}
for all $\varphi\in W^{1,2}(\O)$ with $\varphi=0$ on $\S$.
 Moreover the regularity of $f$, $f\in 
C^\infty(\bar \Omega \backslash \Gamma)$ follows from standard linear elliptic PDE theory.

From the fact $\bar \n^2 V_a= -KV_a \bar g$, we see \begin{eqnarray}\label{fact1}
&&\ode V_a+K(n+1)V_a=0 ,\quad \bar \Delta V_a\bar g-\bar \n^2 V_a+V_a\overline{{\rm Ric}}=0.
\end{eqnarray}
By using Green's formula, \eqref{mixed2} and \eqref{fact1}, we have \begin{eqnarray}\label{first}ˆ\int_\O V_a d\O&=& \int_\O V_a(\ode f+K(n+1)f)- (\ode V_a+K(n+1)V_a) f d\O\nonumber
\\&=&\int_{\p\O} V_a f_\nu- (V_a)_\nu f dA\nonumber
\\&=&\int_{\S}V_a f_{\nu} dA+\int_{T} V_a f_{\bar N}-f(V_a)_{\bar N} dA\nonumber
\\&=&\int_{\S} V_a f_{\nu} dA.
\end{eqnarray}
Using H\"older's inequality for the RHS of \eqref{first}, we have
\begin{eqnarray}\label{first1}
\left(\int_\O V_ad\O\right)^2 \leq \int_{\S} V_aHf_{\nu}^2 dA\int_\S \frac{V_a}{H} dA.
\end{eqnarray}

Next, we use formula \eqref{qx} in our situation with $V=V_a$. Because of \eqref{fact1} and \eqref{mixed2}, formula \eqref{qx} gives
\begin{eqnarray}\label{qx2}
\frac{n}{n+1}\int_\O V_a d\O&=&  \int_\O V_a (\ode f+K(n+1)f)^2 d\O-\frac{1}{n+1} \int_\O V_a (\ode f+K(n+1)f)^2d\O
\nonumber
\\
&\geq &\int_\O V_a \left((\ode f+K(n+1)f)^2-|\bar \n^2 f+K(n+1)f\bar g|^2\right)d\O\nonumber
 \\&=&\int_{\S} V_a Hf_{\nu}^2 dA +\int_T\left(h^{\p B}-\frac{(V_a)_{\bar N}}{V_a}g^{\p B}\right)\left(\n f-\frac{\n V_a}{V_a}f, \n f-\frac{\n V_a}{V_a}f\right)dA\nonumber
 \\&=&{\rm I} +{\rm II}.
\end{eqnarray}

We claim that  \begin{eqnarray}\label{bdrycurv}
h^{\p B}-\frac{(V_a)_{\bar N}}{V_a}g^{\p B}=0, \hbox{ on }T\subset \p B,
\end{eqnarray}
which implies  that term II vanishes.
We take the hyperbolic case for instance. First, $\p B$ is umbilical in $\hh^{n+1}$ with all principal curvatures $\coth R=\frac{1+R_{\rr}^2}{2R_{\rr}}$. Second, 
since $\bar N=\frac{1-R_\rr^2}{2}\frac{x}{R_{\rr}}$ and $V_{a}=\frac{2\<x, a\>}{1-|x|^2}$, a  direct computation gives
\begin{eqnarray*}
\frac{(V_a)_{\bar N}}{V_a}\bigg|_{|x|=R_{\rr}}= \left\<\bar \n^{\rr}\log \left(\frac{2\<x, a\>}{1-|x|^2}\right),  \frac{1-R_\rr^2}{2}\frac{x}{R_{\rr}}\right\>\bigg|_{|x|=R_{\rr}}=\frac{1+R_{\rr}^2}{2R_\rr}.\end{eqnarray*}
Thus \eqref{bdrycurv} follows for the hyperbolic case. For other two cases \eqref{bdrycurv} follows similarly.

Taking account of the above information in \eqref{qx2}, we obtain
\begin{eqnarray}\label{reillyApp}
&&\frac{n}{n+1} \int_\O V_a d\O\ge   \int_{\S} V_aHf_\nu^2dA .
\end{eqnarray}
Combining \eqref{first1} and \eqref{reillyApp}, we conclude \eqref{HKR}.

We are remained to consider the equality case.
If $\Sigma$ a spherical cap which meets $\p B$ orthogonally,   the  Minkowski formula \eqref{Mink1'} implies that equality in \eqref{HKR} holds, for $\S$ has constant mean curvature.
Conversely, if equality in \eqref{HKR} holds, then equality in \eqref{qx2} holds, which implies that 
 $\bar \n^2 f+Kf\bar g=0$ holds in $\O$.
Restricting this equation on $\S$, in view of $f=0$ on $\S$ we know that $\S$ must be umbilical. Thus it is  a spherical cap and $\Omega$ is the intersection of two geodesic balls. It is easy to show the contact angle must be $\frac{\pi}{2}$. Indeed, we have an explicit form for $f$:
\begin{equation*}
f(x)=\left\{
\begin{array}{lll}\frac{1}{2(n+1)}d_p(x)^2+A, &K=0,
\\ A\cosh d_p(x)-\frac{1}{n+1}, &K=-1, 
\\ A\cos d_p(x)+\frac{1}{n+1}, &K=1,
\end{array}
\right.
\end{equation*}
where $p\in\bar{M}^{n+1}(K)$, $A\in\rr$ and $d_p$ is the distance function from $p$. From the boundary condition, we see $\bar g(\bar \n f, \bar N)=0$ on $\Gamma= \S\cap T$, and then  $\bar g(\bar \n d_p, \bar N)=0$. This implies that these two intersecting geodesic balls are perpendicular. The proof is completed.
\end{proof}

As an application we  give an integral geometric proof of the Alexandrov Theorem, which was obtained by  Ros-Souam in \cite{RS} by using the method of moving plane. Our proof has the same flavor of Reilly \cite{Reilly} and Ros \cite{Ros}, see also \cite{QX, LX}.
\begin{theorem}
Let $x: M\to \bar \mm^{n+1}(K)$ be an embedded smooth CMC hypersurface into $B$ which  meets $B$ orthogonally. Assume $\S$ lies in a half ball. Then $\S$ is either a spherical cap or part of  totally geodesic hyperplane.
\end{theorem}
\begin{remark}
The condition that $\S$ lies in a half ball cannot be removed because there are other embedded CMC hypersurfaces, like the Denaulay hypersurfaces in a ball which  meets $\p B$ orthogonally which does not lie in a half ball.
\end{remark}
\begin{proof}We take the hyperbolic case for instance.

We claim first that the  constant mean curvature $H$ is non-negative. 
To prove this claim, let the totally geodesic hyperplane $\{\<x, a\>=0\}$ move upward along $a$ direction along the totally geodesic foliation of $\hh^{n+1}$, until it touches $\S$ at some point $p$ at a first time. It is clear that $H=H(p)\ge 0$.
{{If $H=0$, then the boundary point lemma or the interior maximum principle 
implies that $\S$ must be some totally geodesic hyperplane.}}

Next we assume $H>0$.
In this case the two Minkowski formulae \eqref{Mink1'} and \eqref{Mink2'} yield
\begin{eqnarray*} 
(n+1)  \int_\O V_a dA&=&\int_\S \bar g(X_a, \nu)dA=\frac{1}{H}\int_\S H\bar g(X_a, \nu)dA
\\&=&\frac{n}{H}\int_\S V_a dA=n\int_\S\frac{V_a}{H} dA.
\end{eqnarray*}
The above equation means, for the constant mean curvature hypersurface $\S$, the Heintze-Karcher-Ros inequality is indeed an equality. By the classification of equality case in \eqref{HKR},
we conclude that $\S$ must be  a spherical cap. 
The proof is completed.
\end{proof}

Using the higher order Minkowski formulae \eqref{Mink2'} and the Heintze-Karcher-Ros inequality \eqref{HKR}, we can also prove the rigidity when  $\S$ has constant higher order mean curvatures or mean curvature quotients as Ros \cite{Ros} and Koh-Lee \cite{KL}.
\begin{theorem}\label{Alex1}
Let $x: M\to \bar \mm^{n+1}(K)$ be an isometric immersion into a ball with free boundary. 
Assume that $\S$ lies in a half ball.
\begin{itemize}
\item[(i)] Assume $x$ is an embedding and has nonzero constant higher order mean curvatures $\s_k$, $1\le k\le n$. Then $\S$  is a spherical cap.
\item[(ii)] Assume $x$  has nonzero constant curvature quotient,  i.e.,
$$\frac{\s_k}{\s_l}=const., \quad \s_l>0, \quad 1\le l<k\le n.$$ Then $\S$  is a spherical cap.
\end{itemize}
\end{theorem}

Note that in Theorem \ref{Alex1} (ii), we do not need assume the embeddedness of $x$, since in the proof we need only use the higher order Minkowski formulae \eqref{Mink2'} (without use of the Heintze-Karcher-Ros inequality), which is true for immersions, see Remark \ref{remMink}. On the other hand, the condition of embeddedness may not be removed in Theorem \ref{Alex1} (i) in view of Wente's counterexample. The proof is similar, we leave it to the interested reader.

\medskip

\appendix
\section{Existence of weak solution of \eqref{mixed2}}\label{App}
In this Appendix we discuss the existence of weak solution of \eqref{mixed2}, namely 
\begin{equation}\label{mixed3}
\left\{
\begin{array}{rccl}
\ode f+K(n+1)f&=&1&
\hbox{ in } \Omega,\\
f&=& 0&\hbox{ on } \S,\\
V_a f_{\bar N}-f(V_a)_{\bar N}&=& 0&\hbox{ on } T 
\end{array}
\right.
\end{equation}
in the weak sense \eqref{mixed-weak}.
Since our Robin boundary condition has a different sign, i.e. $\frac{(V_a)_{\bar N}}{V_a}<0$, we can not apply known results about the existence for mixed boundary problem, for example, \cite{Liebm}. Here we use the Fredholm alternative theorem. In order to use it, we have to show that 
\begin{equation}\label{mixed4}
\left\{
\begin{array}{rccl}
\ode \phi+K(n+1)\phi&=&0&
\hbox{ in } \Omega,\\
\phi &=& 0&\hbox{ on } \S,\\
V_a \phi_{\bar N}-\phi(V_a)_{\bar N}&=& 0&\hbox{ on } T
\end{array}
\right.
\end{equation}
has only the trivial solution $\phi=0$ in $W^{1,2}(\O)$. For a general domain $\O$ we do not know how to prove it. 
Nevertheless, we can prove it for domains under the conditions given in Theorem \ref{HK}. 
\begin{prop}
 Let $x: M\to \bar \mm^{n+1}(K)$ be an embedded smooth hypersurface into $B$ with $\p\Sigma \subset \p B$. 
Assume $\S$ lies in a half ball $B_{a+}=\{V_a\ge 0\}=\{\<x, a\>\ge 0\}.$ If $\S$ has positive mean curvature, then
\eqref{mixed4} has only the trivial solution $\phi=0$.
\end{prop}

\begin{proof}
Let $\phi\in {W}^{1,2}(\O)$ be a weak solution of \eqref{mixed4}, i.e., $\phi=0$ on $\S$ and
\begin{eqnarray*}
\int_\O [\bar g(\bar \n \phi, \bar \n\varphi) - K(n+1)\phi \varphi] \, d\O=\int_T \phi \varphi dA, 
\end{eqnarray*}
for all $\varphi\in W^{1,2}(\O)$ with $\varphi=0$ on $\S$.
The classical elliptic PDE theory gives the regularity $\phi\in  C^\infty(\bar \Omega \backslash \Gamma).$ Now we can use the Reilly type formula, \eqref{qx}, as in the proof of Theorem \ref{HK}. Replacing $f$ by $\phi$ in \eqref{qx}, using \eqref{bdrycurv}, we get
\[
-\int_{\O}V_a|\bar \n^2  \phi+K \phi\bar g|^2d\O
=\int_{\Sigma} V_aH (\phi_\nu)^2 dA.
\]
Since $V_a$ and $H$ is positive, 
it follows  that  
\begin{eqnarray}\label{classif}
\bar \n^2  \phi+K \phi\bar g=0 \hbox{ in }\O
\end{eqnarray} and $\phi_\nu=0$ on $\S$.
From \eqref{classif} and Remark \ref{rem-classi}, we see $\phi$ must be of form
\begin{eqnarray*}
\phi=\sum_{i=0}^{n+1}b_iV_i,
\end{eqnarray*}
for $b_i\in \mathbb{R}$, $i=0,1,\cdots n+1$.
By checking the boundary condition $V_a \phi_{\bar N}-\phi(V_a)_{\bar N}=0$ on $T$, we see $b_0=0$. Moreover, since
$\phi=0$ on $\S$, $\S$ must be the totally geodesic hyperplane through the origin if one of $b_i\neq 0$, which is a contraction to $\S$ having positive mean curvature. We get the assertion.
\end{proof}


With this Proposition one can use the Fredholm alternative to get a unique weak solution of \eqref{mixed3}.

\end{document}